\documentclass[hidelinks,onefignum,onetabnum]{siamart220329}

\usepackage{amsfonts}
\usepackage{graphicx}
\ifpdf
\DeclareGraphicsExtensions{.eps,.pdf,.png,.jpg}
\else
\DeclareGraphicsExtensions{.eps}
\fi

\newsiamremark{remark}{Remark}
\newsiamremark{hypothesis}{Hypothesis}
\crefname{hypothesis}{Hypothesis}{Hypotheses}
\newsiamthm{claim}{Claim}

\headers{A Rellich-type theorem for a junction of stratified media}{S. Al Humaikani, A.-S. Bonnet-Ben Dhia, S. Fliss, and C. Hazard}

\title{A Rellich-type theorem for the Helmholtz equation in a junction of stratified media\thanks{Submitted to the editors 04/16/2025.}}

\author{Sarah Al Humaikani\thanks{POEMS, CNRS, Inria, ENSTA Paris, Institut Polytechnique de Paris, 91120 Palaiseau, France (\email{sarah.al-humaikani@ensta.fr}, \email{anne-sophie.bonnet-bendhia@ensta.fr}, \email{sonia.fliss@ensta.fr}, \email{christophe.hazard@ensta.fr}).}
	\and Anne-Sophie Bonnet-Ben Dhia\footnotemark[2]
	\and Sonia Fliss\footnotemark[2]
	\and Christophe Hazard\footnotemark[2]}

\usepackage{amsopn}

\usepackage{amsfonts}
\usepackage{amssymb}
\usepackage{graphicx}
\usepackage{dsfont}
\usepackage{amsopn}
\usepackage{stmaryrd} 
\usepackage{tikz}
\usepackage{pgfplots}

\pgfplotsset{every axis/.append style={
		axis x line=middle,   
		axis y line=middle,   
		axis line style={->,color=black}, 
		xlabel={$x$},
		ylabel={$y$},
}}
\usetikzlibrary{decorations.pathmorphing}
\tikzset{snake it/.style={decorate, decoration=snake}}
\usetikzlibrary{shapes}
\usetikzlibrary{patterns,patterns.meta,calc,intersections,decorations.markings}
\usetikzlibrary{positioning}

\usepgfplotslibrary{fillbetween}
\usepgflibrary{shadings}
\DeclareMathOperator*{\esssup}{ess\,sup}

\usepackage{array}
\newcolumntype{M}[1]{>{\centering\arraybackslash}m{#1}}
\newcolumntype{L}[1]{>{\raggedleft\arraybackslash}m{#1}}
\makeatletter
\newcommand{\thickhline}{%
	\noalign {\ifnum 0=`}\fi \hrule height 1pt
	\futurelet \reserved@a \@xhline
}

\newcolumntype{"}{@{\hskip\tabcolsep\vrule width 1pt\hskip\tabcolsep\hspace{-6pt}}}
\newcolumntype{[}{@{\vrule width 1pt\hspace{6pt}}} 
\newcolumntype{]}{@{\hspace{6pt}\vrule width 1pt}} 
\newcommand{\thickcline}[1]{%
	\@thickcline #1\@nil%
}

\def\@thickcline#1-#2\@nil{%
	\omit
	\@multicnt#1%
	\advance\@multispan\m@ne
	\ifnum\@multicnt=\@ne\@firstofone{&\omit}\fi
	\@multicnt#2%
	\advance\@multicnt-#1%
	\advance\@multispan\@ne
	\leaders\hrule\@height1pt\hfill
	\cr
	\noalign{\vskip-1pt}%
}
\makeatother

\newcommand{\C}{\mathbb{C}}
\newcommand{\calC}{\mathcal{C}}
\renewcommand{\d}{\mathrm{d}}
\newcommand{\rmD}{\mathrm{D}}
\newcommand{\calD}{\mathcal{D}}
\newcommand{\e}{{\rm{e}}}
\newcommand{\E}{\textsc{e}}
\newcommand{\calF}{\mathcal{F}}
\renewcommand{\H}{\mathcal{H}}
\newcommand{\rmH}{\mathrm{H}}
\renewcommand{\i}{\mathrm{i}}
\newcommand{\J}{\textsc{j}}
\renewcommand{\k}{\textsc{k}}
\newcommand{\km}{\textsc{k}_{-}}
\newcommand{\kp}{\textsc{k}_{+}}
\newcommand{\kpm}{\textsc{k}_\pm}
\newcommand{\N}{\textsc{n}}
\newcommand{\R}{\mathbb{R}}
\newcommand{\x}{\textsc{x}}
\newcommand{\y}{\textsc{y}}
\newcommand{\W}{\textsc{w}}

\renewcommand{\Im}{\mathrm{Im}}
\renewcommand{\Re}{\mathrm{Re}}

\newcommand{\HH}{\widehat{\mathcal H}}
\newcommand{\Hphi}{\widehat\varphi}
\newcommand{\Hpsi}{\widehat\psi}
\newcommand{\Hu}{\widehat u}
\newcommand{\Tphi}{\widehat\varphi}
\newcommand{\Tu}{\widehat u}
\renewcommand{\widetilde}[1]{{\widehat{#1}}}

\begin{document}





\maketitle
\begin{abstract}
  We prove that there are no non-zero square-integrable solutions to a two-dimension\-al Helmholtz equation in some unbounded inhomogeneous domains which represent junctions of stratified media. More precisely, we consider domains that are unions of three half-planes, where each half-plane is stratified in the direction orthogonal to its boundary. As for the well-known Rellich uniqueness theorem for a homogeneous exterior domain, our result does not require any boundary condition. Our proof is based on half-plane representations of the solution which are derived through a generalization of the Fourier transform adapted to stratified media. A byproduct of our result is the absence of trapped modes at the junction of open waveguides as soon as the angles between branches are greater than $\pi/2$.
\end{abstract}



\section{Introduction}
\subsection{Motivation and main result}\label{setting pb}
In the present paper, we study the possible existence of square-integrable solutions to a two-dimensional  Helmholtz equation with varying wavenumber in unbounded media that can be interpreted as  \emph{junctions of stratified half-planes}.

We first introduce what we call hereafter a \emph{stratified half-plane}, defined in a local coordinate system $(\x,\y)$ by
\begin{equation}
	\rmH := \{ (\x,\y) \in \R^2 \mid \y > 0 \} = \R \times \R^+,
	\label{def half plane}
\end{equation}
where we denote $\R^+ := (0,+\infty)$. This half-plane is called stratified in the sense that it is associated with a real-valued wavenumber function $\k$ that only depends on the $\x$-coordinate, i.e., $\k = \k(\x)$. In other words, the stratification is orthogonal to the boundary of $\rmH$. In this paper, the direction parallel to the $\x$-axis is called \emph{transverse}, whereas the perpendicular one (along the $\y$-axis) is called \emph{longitudinal}. Moreover we assume that $\k$ is bounded and becomes constant outside a bounded interval, that is, there exist positive numbers $(\km,\kp)$ and a pair $(\x^-,\x^+) \in \R^2$ with $\x^- \leq \x^+$ such that
\begin{equation}
	\forall \x \in \R, \quad 
	\k(\x) = \km \text{ if } \x < \x^-
	\quad\text{and}\quad 
	\k(\x) = \kp \text{ if } \x > \x^+.
	\label{eq:assump-K}
\end{equation}

\begin{figure}[t]
	\centering
		\centering
		\centering\includegraphics[height=5cm]{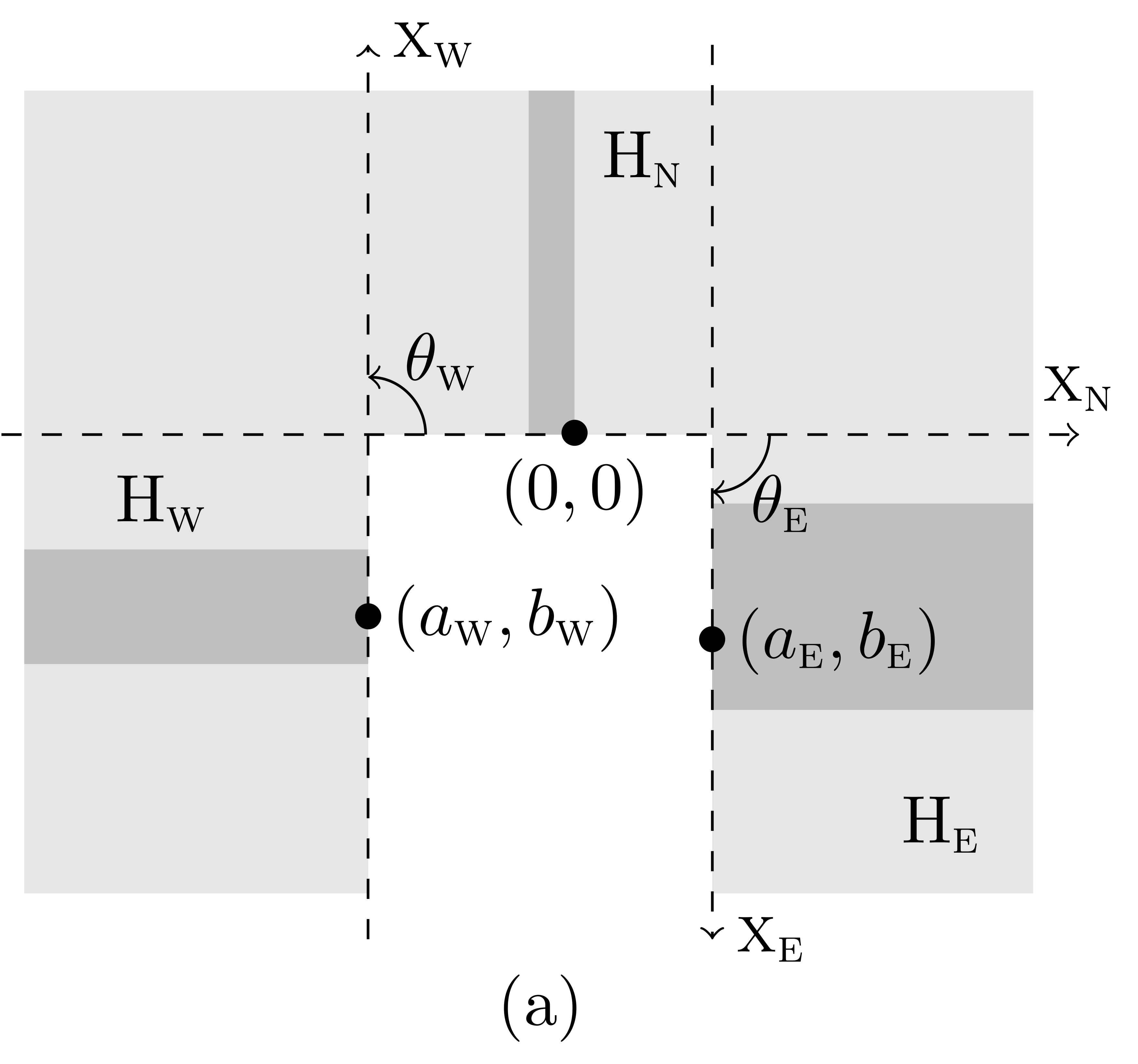}
		\hspace{.02\columnwidth}
\includegraphics[height=5cm]{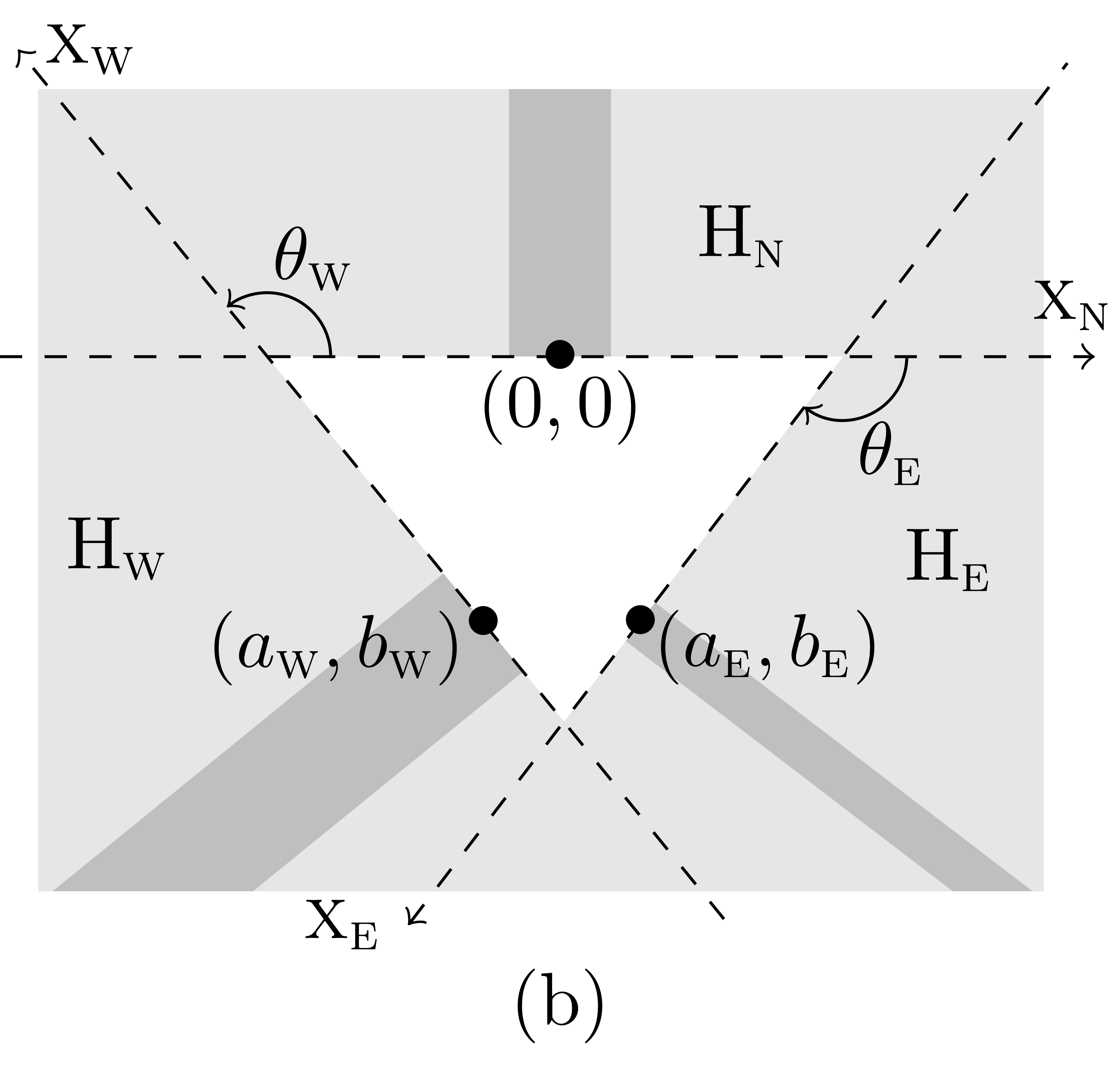}
\caption{The two considered configurations. (a): right-angle, (b): general angle. The light gray represents various constant values of $k$, whereas the dark gray stands for possible stratifications.}
\label{fig:two-config}  
\end{figure}

The media we are interested in, illustrated in \cref{fig:two-config}, represent \emph{junctions} of three such stratified half-planes, denoted $\rmH_\W$, $ \rmH_\N$ and $ \rmH_\E$ (where the indices stand respectively for west, north and east). Each of them is associated with local coordinates $(\x_\J,\y_\J)$ and a wavenumber function $\k_\J = \k_\J(\x_\J)$ for $\J \in \{\W,\N,\E\}$, as described above. We choose a global coordinate system $(x,y)$ which coincides with the local system of the northern half-plane, whereas for each of the two other half-planes, the local coordinates $(\x_\J,\y_\J)$ are obtained from the global ones $(x,y)$ by a rotation of angle $\theta_\J$ centered at a given point $(a_\J,b_\J)$, for $\J \in \{\W,\E\}$. More precisely, we have
\begin{equation}
	\begin{pmatrix}
		\x_\N \\ \y_\N
	\end{pmatrix} = \begin{pmatrix}x \\ y\end{pmatrix}\quad\text{and}\quad
	\begin{pmatrix}
		\x_\J \\\y_\J
	\end{pmatrix}=\begin{pmatrix}
		\cos\theta_\J & \sin\theta_\J \\ -\sin\theta_\J & \cos\theta_\J
	\end{pmatrix}\begin{pmatrix} x -a_\J \\ y-b_\J\end{pmatrix}\text{ for }\J\in\{\W,\E\},
	\label{eq: coord locales}
\end{equation}
where $\theta_\W \in(0,+\pi)$ and $\theta_\E \in(-\pi,0)$. We denote by $\Omega$ the union of the three half-planes, i.e.,
\begin{equation*}
	\Omega := \rmH_\W \cup \rmH_\N \cup \rmH_\E,
\end{equation*}
and we consider the wavenumber function defined in $\Omega$ by
\begin{equation}
	\forall \J \in \{\W,\N,\E\}, \ \forall (x,y) \in \rmH_\J, \quad k(x,y) := \k_\J\big( \x_\J(x,y)\big),
\end{equation}
where the function $\x_\J(x,y)$ refers to the mappings \cref{eq: coord locales} between global coordinates $(x,y)$ and local ones $(\x_\J,\y_\J)$. Of course, for this definition to make sense, it is necessary that for all points $(x,y)$ in the intersection of two half-planes, both possible definitions of $k(x,y)$ coincide. This clearly implies that $k$ must be constant in each intersection, hence that the stratifications must be located outside these intersections. This also implies that $|\theta_\J| \geq\pi/2$ for $\J\in\{\W,\E\}$, since the stratifications are orthogonal to the boundaries of the $\rmH_\J$'s. In other words, we assume that
\begin{equation}
	\theta_\W \in [+\pi/2,+\pi) \quad\text{and}\quad \theta_\E \in(-\pi,-\pi/2].
	\label{eq: theta W E}
\end{equation}
These conditions lead us to consider the two configurations depicted in \cref{fig:two-config}. The so-called \emph{right-angle configuration} on the left-hand side corresponds to the limit case where $(\theta_\W,\theta_\E ) = (+\pi/2,-\pi/2)$, for which $\Omega$ is the complement of a semi-infinite strip. In the other case that we call the \emph{general angle configuration}, we have $(\theta_\W,\theta_\E ) \neq (+\pi/2,-\pi/2)$, which shows that $\Omega$ is now the complement of a triangle.

The aim of the present paper is to prove the following statement.

\begin{theorem}\label{Theorem: main}
	Considering $\Omega\subset \R^2$ and a function $k\in L^\infty(\Omega)$ as described above, if $u\in L^2(\Omega)$ satisfies the Helmholtz equation
	\begin{equation}
		-\Delta u - k^2 u = 0 \text{ in }\Omega \label{eq Helmholtz}
	\end{equation}
	in the distributional sense, then $u=0$ in $\Omega$.
\end{theorem}

\Cref{Theorem: main} can be deduced from existing results in some particular cases which are illustrated in \cref{fig:examples-known}. The simplest situation is the case where $k$ is constant in the whole domain $\Omega$. In the general angle configuration (case (a) of \cref{fig:examples-known}),  \cref{eq Helmholtz} implies that $u$ satisfies the Helmholtz equation outside a disk, then the well-known Rellich's uniqueness theorem \cite{Rellich-1943} tells us that $u$ vanishes outside this disk, hence also in $\Omega$ according to the unique continuation principle. In the right-angle configuration (case (b) of \cref{fig:examples-known}), Rellich's theorem no longer applies, but one can use instead the uniqueness theorem in \cite{Bonnet-Fliss-Hazard-Tonnoir-2011}, which is exactly the statement of \cref{Theorem: main} when $k$ is constant and $\Omega$ is a conical domain with vertex angle greater than $\pi$, that is,
$\Omega := \{(x,y) \in \R^2 \mid y > -|x| \tan\theta \}$ with $\theta \in (0,\pi/2)$. The latter result (together with the unique continuation principle) also yields \cref{Theorem: main} for some cases of non-constant $k$, as soon as $k$ remains constant in a conical subdomain of $\Omega$ with vertex angle greater than $\pi$. Such a condition is satisfied for instance in the general angle configuration when $k$ is constant in one of the three half-planes (case (c) of \cref{fig:examples-known}), or in $\rmH_\W \cup \rmH_\N$ (or $\rmH_\E \cup \rmH_\N$) in the right-angle configuration (case (d) of \cref{fig:examples-known}).

\begin{figure}[t]
	\centering
\includegraphics{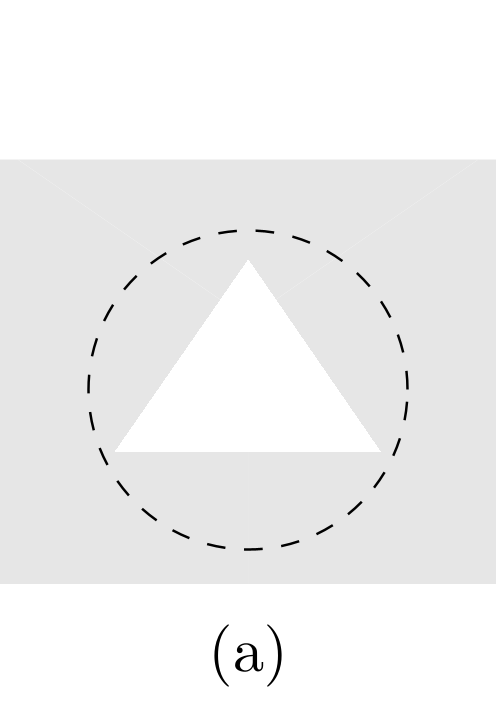}
\hspace{.02\columnwidth}
\includegraphics{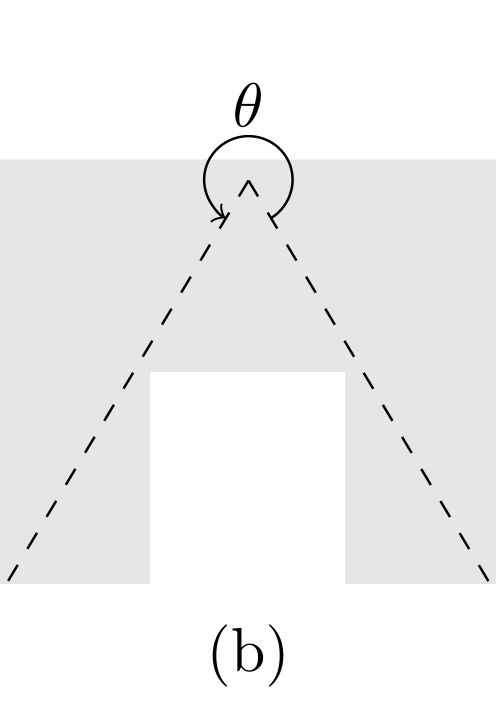}
	\hspace{.02\columnwidth}
	\centering\centering\includegraphics{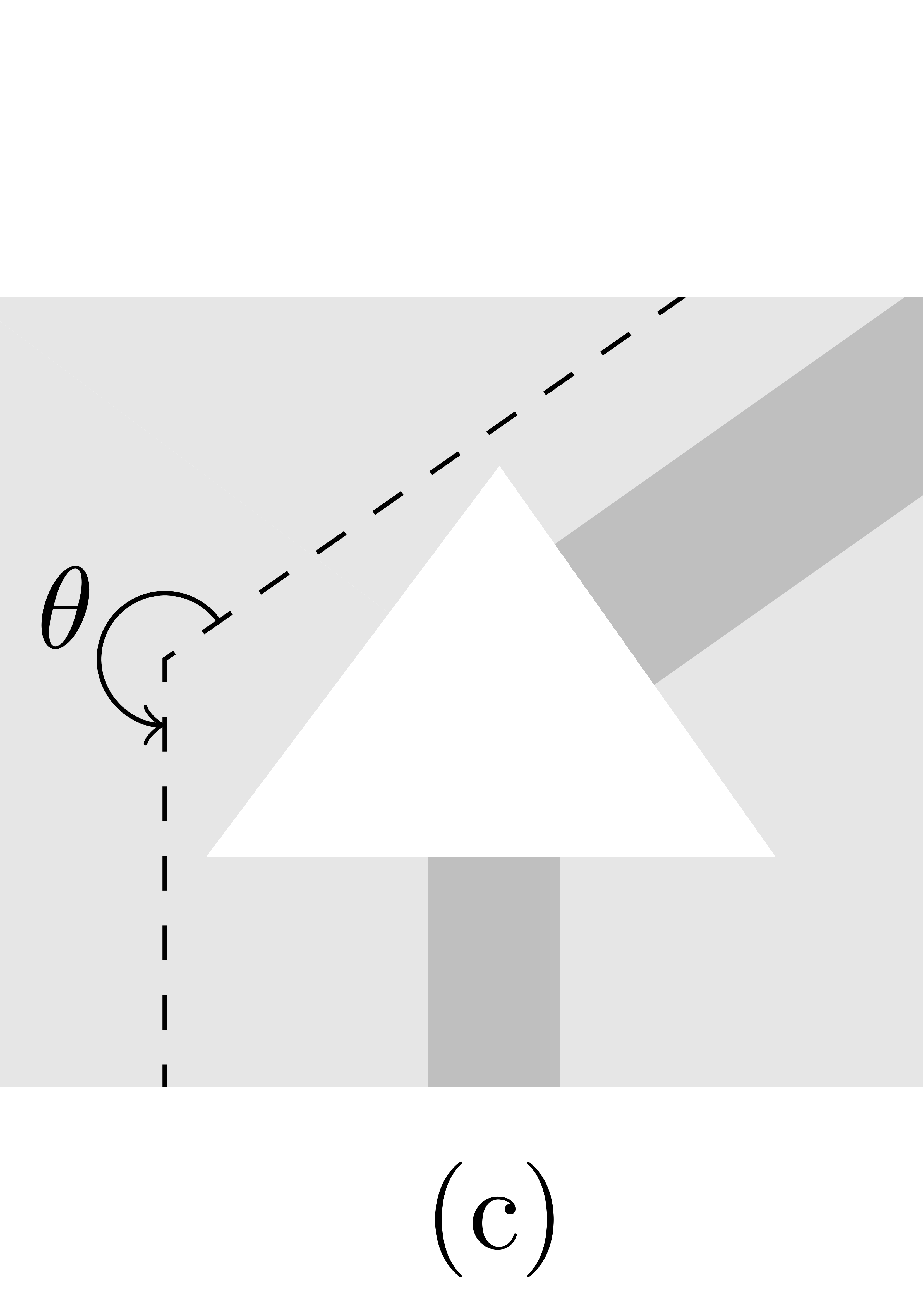}
		\hspace{.02\columnwidth}
	\centering
	\includegraphics{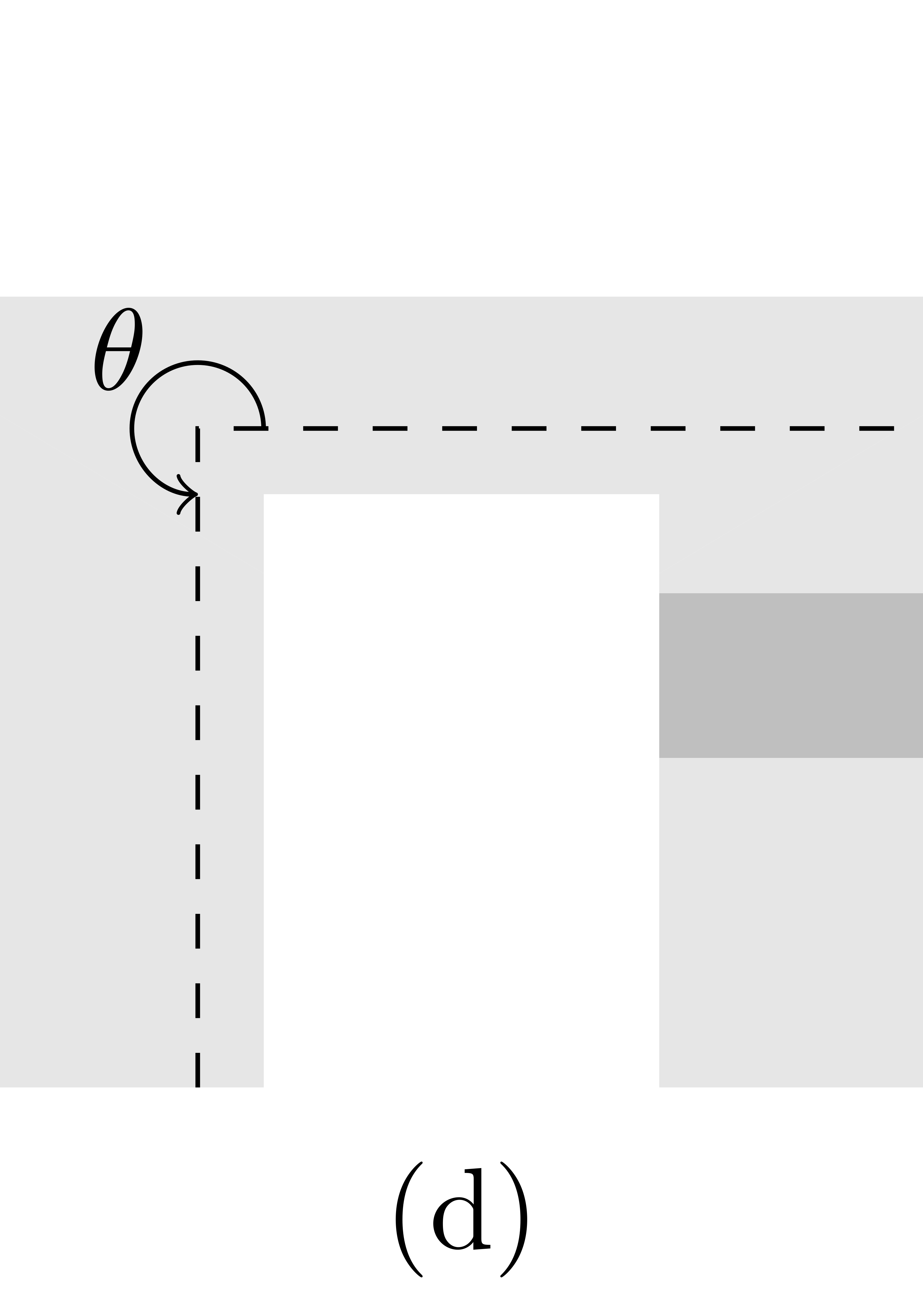}
	\caption{Examples of situations where \cref{Theorem: main} can be deduced from existing results.} 
	\label{fig:examples-known}
\end{figure}

Let us point out that no boundary condition is required in \cref{Theorem: main}, nor in the uniqueness results \cite{Bonnet-Fliss-Hazard-Tonnoir-2011,Rellich-1943} mentioned above. Indeed, these results express that in some particular unbounded domains, the $L^2$ framework prohibits the existence of non-trivial solutions to the Helmholtz equation. In other words, whatever the content of $\R^2 \setminus \Omega$ (i.e., the white semi-infinite strip or triangle of \cref{fig:two-config}), \cref{Theorem: main} applies.  \Cref{fig:examples-new} shows some possible examples covered by \Cref{Theorem: main}: case (a) shows a junction of three stratified media, whereas case (b) represents a more exotic junction of three stratified media together with a periodic medium. 
From a physical point of view, \Cref{Theorem: main} ensures the absence of \emph{trapped modes}, that is, localized vibrations of an unbounded medium (provided the unique continuation principle in $\R^2 \setminus \Omega$ holds true). From a spectral point of view, if we consider the selfadjoint operator which describes the dynamics of such a junction, the absence of trapped modes amounts to the absence of eigenvalues embedded in the essential spectrum $[0,+\infty)$ of this operator. 

The possible existence of trapped modes at the junction of open waveguides initially motivated this paper. By \emph{open waveguide}, we mean a stratified medium which is unbounded in the transverse direction and such that time-harmonic waves can propagate without attenuation in the longitudinal direction while their energy remains localized in the transverse direction, near the \emph{core} of the waveguide (the dark gray areas in \cref{fig:two-config}). It is known that trapped modes do not exist in uniform straight open waveguides, nor for local perturbations of such waveguides  \cite{Bonnet-Dakhia-Hazard-Chorfi-2009,DeBievre-Pravica-1992,Hazard-2015,Weder-1988,Weder-1991}, nor in straight junctions of such waveguides \cite{Bonnet-Goursaud-Hazard-2011}. Thanks to \cite{Bonnet-Fliss-Hazard-Tonnoir-2011}, we also know that such modes do not exist in curved open waveguides either, provided that the core of the waveguides and all other possible inhomogeneities can be encased in a cone of opening angle smaller than $\pi$ (as in case (c) of \cref{fig:examples-known}). The present paper provides other configurations where this result holds true. In short, \cref{Theorem: main} implies the absence of trapped modes at the junction of open waveguides, provided the angles between branches are all greater or equal to $\pi/2$.

The above mentioned results, as well as others in the same vein (see e.g. \cite{Vesalainen-2014}) could suggest a more general result stating the absence of trapped modes in any unbounded domain. This is not true. Such trapped modes are known to occur in \emph{closed waveguides}, i.e., unbounded tubes with bounded transverse section, about which there exists an abundant literature (see for instance the review papers \cite{Linton-McIver-2007, Pagneux-2013}). We also mention the recent result of \cite{Krejcirik-Lotoreichik-2024} which proves the existence of an embedded eigenvalue for a so-called quasi-conical domain.
\begin{figure}[t]
	\centering
	\includegraphics[height=3.8cm]{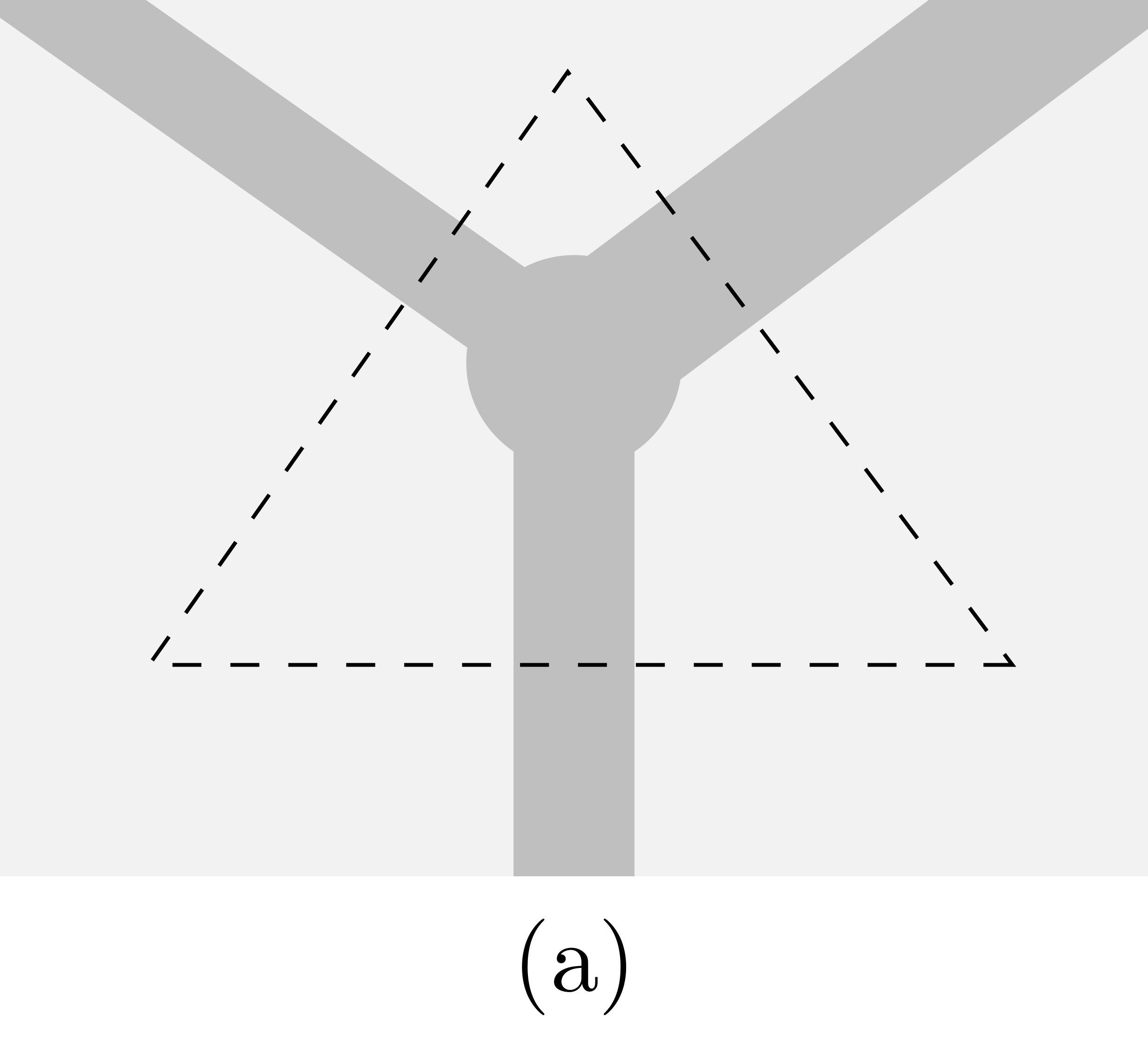}
		\centering
		\hspace{.05\columnwidth}
		\centering
		\includegraphics[height=3.8cm]{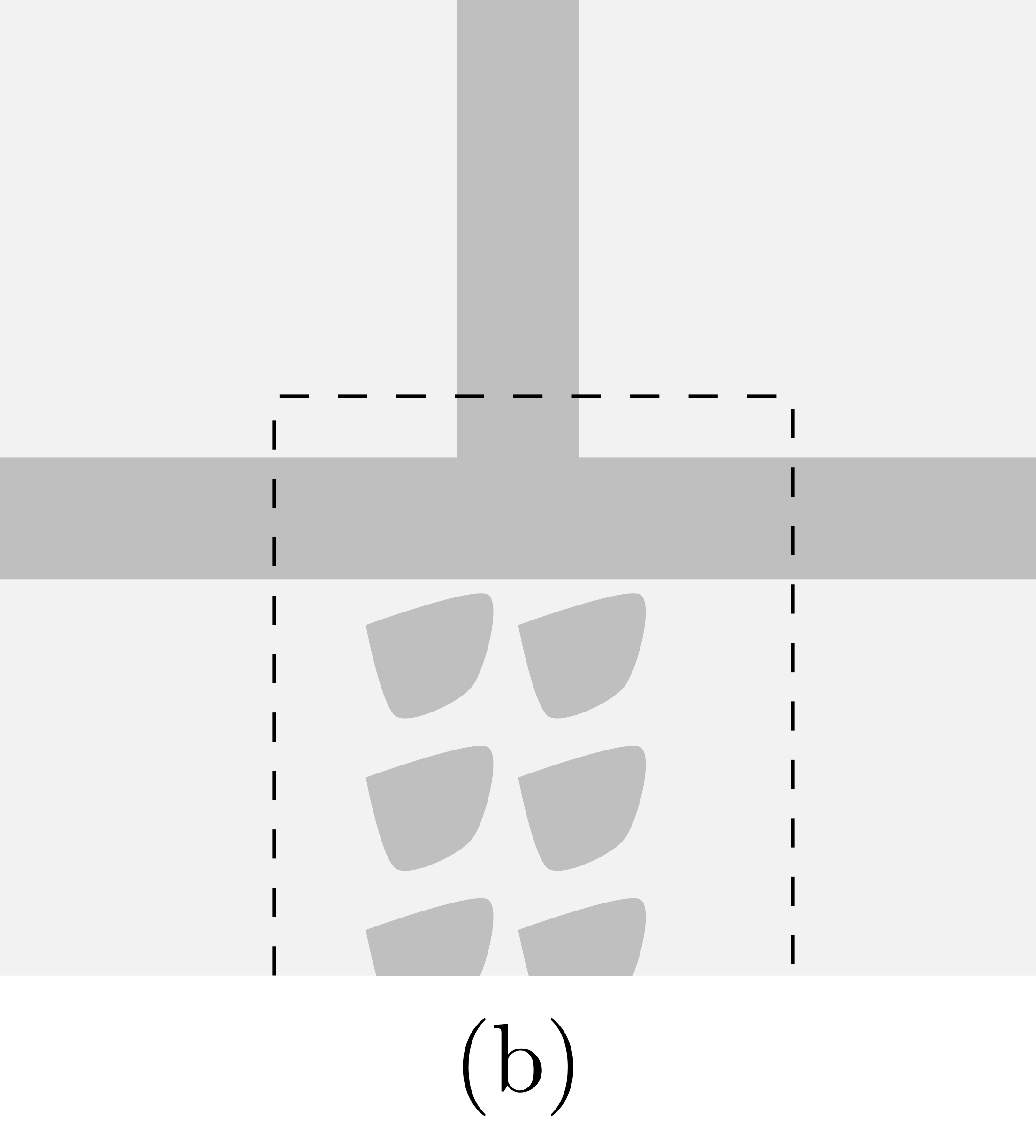}
		\caption{Examples of situations where \cref{Theorem: main} applies.} 
		\label{fig:examples-new}
\end{figure}

The proof of \cref{Theorem: main} follows exactly the same lines as in  \cite{Bonnet-Fliss-Hazard-Tonnoir-2011}. The key idea, exploited also in \cite{Bonnet-Dakhia-Hazard-Chorfi-2009,Bonnet-Goursaud-Hazard-2011,Hazard-2015}, dates back to the pioneering work of \cite{Weder-1988}. There are two main ingredients in the proof. 

{\rm (i)} On the one hand, we use a \emph{half-plane representation} of $u$, whose construction relies on separation of local space coordinates\footnote{Rellich's uniqueness theorem \cite{Rellich-1943} also relies on separation of variables. But the separation is made with polar (or spherical) coordinates instead of Cartesian coordinates. Let us mention that even in the very particular sub-case of our problem where $k$ only depends on the polar angle (which may happen in the situations considered in \cref{sec: two-layer}), separation of polar coordinates cannot be used to prove \cref*{Theorem: main}.} in each of the three half-planes. In a homogeneous medium, the usual Fourier transform along the transverse $\x$-direction provides such a representation, for it diagonalizes the transverse component of the Helmholtz operator, that is, $-\partial^2_\x$. This is no longer possible in a stratified medium, since we cannot express conveniently the usual Fourier transform of $k^2u$ if $k^2$ is not constant. Fortunately we can construct a \emph{generalized Fourier transform} which diagonalizes $-\partial^2_\x-k^2$ for non-constant $k^2$ and then yields a half-plane representation (see \cite{Ott-2017,Santosa-Magnanini-2001} for examples of the use of the generalized Fourier transform in open waveguides). From a physical point of view, the latter representation is interpreted as a modal decomposition of $u$ on a continuous family of modes of the stratified medium. There are two categories of modes depending on their behaviour in the longitudinal direction : evanescent or travelling. Since we are concerned with localized (square-integrable) solutions, the modal components of $u$ associated with travelling modes must vanish. 

{\rm (ii)} On the other hand, the second ingredient of the proof is an \emph{analyticity argument} with respect to the Fourier variable, which ensures that the modal components extend to an analytic function of the Fourier variable. Thanks to the isolated zeros principle, the fact that the travelling components of $u$ vanish implies that the evanescent ones also vanish, hence that $u=0$, which completes the uniqueness proof. 

The paper is organized as follows. We start in \cref{section : homogeneous case} by recalling the sketch of the proof given in \cite{Bonnet-Fliss-Hazard-Tonnoir-2011} for a homogeneous medium in the right-angle configuration. Then, for the sake of clarity, we focus on a particular case of stratified media composed of only two layers, for which more explicit calculations can be made. \Cref{sec: HPR two-layered} is devoted to the half-plane representation in such a two-layered medium, after the presentation of the generalized Fourier transform in this context. In \cref{sec: two-layer}, we give the proof of \cref{Theorem: main} for a junction of two-layered media, considering successively the right-angle and the general angle configurations. The case of general stratified media is finally dealt with in \cref{section : general case}.

\subsection{Sketch of the proof for a homogeneous medium}\label{section : homogeneous case}
As our proof of \cref{Theorem: main} is a generalization of the one presented in \cite{Bonnet-Fliss-Hazard-Tonnoir-2011} in the case of a homogeneous medium, we begin here by recalling the basic ideas of the latter proof. For the sake of simplicity, we restrict ourselves to the right-angle configuration, i.e., $(\theta_\W,\theta_\E)=(+\pi/2,-\pi/2)$, see \cref{fig:two-config} (strictly speaking, this case is not considered in \cite{Bonnet-Fliss-Hazard-Tonnoir-2011}, but the proof is easily adapted). For greater clarity, we omit most of the technical details concerning the functional framework.

Hence, in this subsection, function $k(x,y)$ is constant, simply denoted $k$.

\textbf{First step}. The core ingredient of the proof is a half-plane representation of a function $u\in L^2(\rmH)$ which satisfies
\begin{equation*}
	-\Delta u -  k ^2 u = 0 \quad\text{in } \rmH,
\end{equation*}
where $\rmH$ is the half-plane defined in \cref{def half plane}. This representation is obtained by separation of space variables, thanks to the usual Fourier transform in the $\x$-direction, defined by
\begin{equation}
	\forall\xi\in\R,\quad F f(\xi) := \int_\R f(\x)\,\e^{-\i \xi \x}\,\d\x.
	\label{def: usual Fourier}
\end{equation}
Let $\Tu$ denote the partial Fourier transform of $u$  in the $\x$-direction, that is,   
$\Tu(\xi,\y) := F[u(\cdot,\y)](\xi)$. Applying $F$ to the Helmholtz equation shows that for (almost) every $\xi \in \R$,
\begin{equation*}
	-\partial_\y^2 \Tu(\xi,\cdot) - ( k^2-\xi^2)\, \Tu(\xi,\cdot) = 0 
	\quad\text{in }(0,+\infty).\label{HP problem hom four}
\end{equation*}
Thus there exist some functions $A(\xi)$ and $B(\xi)$ such that
\begin{equation*}
	\Tu(\xi,\y)=A(\xi)\,\e^{-\sqrt{\xi^2- k^2}\,\y} + B(\xi)\,\e^{\sqrt{\xi^2- k^2}\,\y},
\end{equation*}
where we use for instance the convention $\sqrt z := \i\sqrt{-z}$ if $z<0$. As $F$ is unitary from $L^2(\R_\x)$ to  $L^2(\R_\xi)$, we infer from the initial assumption $u\in L^2(\rmH)$ that $\Tu\in L^2(\R_\xi\times(0,+\infty))$, which implies that 
for (almost) every $\xi \in \R$, $\Tu(\xi,\cdot)$ belongs to $L^2(0,+\infty)$. As a consequence, we have $A(\xi)=0$ for $|\xi|< k $ and $B(\xi)=0$ for $\xi\in\R$, which yields 
\begin{equation}
	\forall (\xi,\y) \in \R\times[0,+\infty), \quad
	\Tu(\xi,\y)=
	\left\{ \begin{array}{ll} 
		0 & \text{if } |\xi|< k, \\
		\Tu(\xi,0)\,\e^{-\sqrt{\xi^2- k^2}\,\y} & \text{if } |\xi| > k.
	\end{array}\right.
	\label{eq: hom rep Four}
\end{equation}
Using the inverse Fourier transform, we finally obtain the so-called half-plane representation of $u$, which provides an expression of $u$ at any point of $\rmH$ from the Fourier transform of its trace on the boundary of $\rmH$:
\begin{equation}
	\forall (\x,\y)\in \rmH,\quad u(\x,\y) = \frac{1}{2\pi} \int_{|\xi|>k } \e^{\i \xi \x} \, \e^{-\sqrt{\xi^2- k^2}\,\y}\,\Tu(\xi,0)\ \d \xi. \label{cas hom/ rep demi plan}
\end{equation}

\textbf{Second step}. Let us show now how to use this representation to prove \cref{Theorem: main} in the right-angle configuration, i.e., when the relations between the local and global coordinate systems \cref{eq: coord locales} simplify as
\begin{equation}
	\begin{pmatrix} \x_\N \\ \y_\N \end{pmatrix} 
	= \begin{pmatrix}x \\ y\end{pmatrix},
	\quad
	\begin{pmatrix} \x_\W \\ \y_\W \end{pmatrix} 
	= \begin{pmatrix}y-b_\W \\ -x+a_\W \end{pmatrix}
	\quad\text{and}\quad
	\begin{pmatrix} \x_\E \\ \y_\E \end{pmatrix} 
	= \begin{pmatrix}-y+b_\E \\ x-a_\E \end{pmatrix}.
	\label{eq: coord locales right angle}
\end{equation}
Suppose that $u \in L^2(\Omega)$ satisfies the Helmholtz equation in $\Omega := \rmH_\W \cup \rmH_\N \cup \rmH_\E$. Hence, we can use the half-plane representation \cref{cas hom/ rep demi plan} for the restriction of $u$ in each of the three half-planes. For $\J \in \{\W,\N,\E\}$, we denote by  $\varphi_\J$ the trace of $u$ on $\partial \rmH_\J$, considered as a function of the local coordinate $\x_\J$, and by $\Tphi_\J$ its Fourier transform with respect to $\x_\J$. We know from \cref{eq: hom rep Four} for $\y = 0$ that $\Tphi_\J(\xi) = 0$ if $|\xi|<k$.

In order to prove that $u=0$, the key argument consists in proving that $\Tphi_\N(\xi)$ extends to an analytic function of $\xi$ in a complex vicinity of the real axis. As $\Tphi_\N(\xi)$ vanishes on the interval $(-k,+k),$ analyticity implies that it vanishes on the whole real axis. The half-plane representation in $\rmH_\N$ then tells us that $u$ vanishes in $\rmH_\N$, so finally also in the whole domain $\Omega$ by virtue of the unique continuation principle.

It remains to prove the analyticity of $\Tphi_\N(\xi)$. To do so, we start from the definition of $\Tphi_\N$ (recall that $\x_\N=x$),
\begin{equation*}
	\Tphi_\N(\xi):=\int_{\R} u(x,0)\,\e^{-\i\xi x}\,\d x
\end{equation*}
and split the integral in three parts by setting 
\begin{equation*}
	\Tphi_\N(\xi)=
	\underbrace{\int_{-\infty}^{a_\W} u(x,0)\,\e^{-\i\xi x}\,\d x}_{\displaystyle =: \Tphi_{\N,\W}(\xi)}
	+ \underbrace{\int_{a_\W}^{a_\E} u(x,0)\,\e^{-\i\xi x}\,\d x}_{\displaystyle =: \Tphi_{\N,0}(\xi)}
	+ \underbrace{\int_{a_\E}^{+\infty} u(x,0)\,\e^{-\i\xi x}\,\d x}_{\displaystyle =: \Tphi_{\N,\E}(\xi)}.
\end{equation*}
Thanks to Morera's theorem\footnote{Morera's theorem tells us that if $f$ is a continuous complex-valued function defined in a given open disk $D$ of the complex plane, such that $\int_T f(\xi) \,\d\xi = 0$ for every triangular path $T \subset D$, then $f$ is analytic in $D$. To apply it to $\Tphi_{\N,0}(\xi)$, we simply use Fubini's theorem.}, we see that the second integral $\Tphi_{\N,0}(\xi)$ extends to an entire function of $\xi$, since $x$ lies in a bounded interval (so that $\e^{-\i\xi x}$ remains bounded for $\xi$ in any bounded complex domain). To deal with $\Tphi_{\N,\W}(\xi)$ and $\Tphi_{\N,\E}(\xi)$, we use the half-plane representation in $\rmH_\W$ and $\rmH_\E$ respectively. As both cases are similar, we focus on $\Tphi_{\N,\W}(\xi)$. Using \cref{eq: coord locales right angle}, the half-plane representation \cref{cas hom/ rep demi plan} becomes
\begin{equation*}
	\forall (x,y)\in \rmH_W,\quad u(x,y) = \frac{1}{2\pi} \int_{|\eta|>k } \e^{\i \eta (y-b_\W)} \, \e^{-\sqrt{\eta^2- k ^2}(-x+a_\W)}\,\Tphi_\W(\eta)\, \d \eta.
\end{equation*}
Taking $y=0$, we obtain
\begin{equation*}
	\Tphi_{\N,\W}(\xi) = 
	\frac{1}{2\pi} \int_{-\infty}^{a_\W} \left( \int_{|\eta|>k } \e^{-\i \eta b_\W} \, \e^{-\sqrt{\eta^2- k ^2}(-x+a_\W)}\,\Tphi_\W(\eta)\, \d \eta \right) \e^{-\i\xi x}\,\d x.
\end{equation*}
Fubini's theorem then yields
\begin{align*}
	\Tphi_{\N,\W}(\xi) 
	& = \frac{\e^{-\i\xi a_\W}}{2\pi} \int_{|\eta|>k } 
	\left( \int_{-\infty}^{a_\W} \e^{(\sqrt{\eta^2- k ^2}-\i\xi)(x-a_\W)}\, \d x \right) 
	\e^{-\i \eta b_\W}\,\Tphi_\W(\eta) \,\d\eta \\
	& = \frac{\e^{-\i\xi a_\W}}{2\pi} \int_{|\eta|>k }
	\frac{\e^{-\i \eta b_\W}\,\Tphi_\W(\eta)}{\sqrt{\eta^2- k ^2}-\i\xi} \,\d\eta.
\end{align*}
Using again Morera's theorem, we conclude that the latter integral extends to an analytic function of $\xi$ in $\C\setminus \i\R$ (since the denominator does not vanish). Proceeding in a similar way on the east side, we have finally proved that $\Tphi_\N(\xi)$ extends to an analytic function of $\xi$ in $\C\setminus \i\R$, which completes the proof.

\begin{remark}
	To ensure that the above approach is rigorous, we need in particular to justify that all the integrals we have used are well defined in the Lebesgue sense. Among other things, we must have $\Tphi_\J \in L^1(\R)$ for each $\J \in \{\W,\N,\E\}$. This is not true in general, since the hypotheses of \cref{Theorem: main} do not provide any regularity assumption on the $\varphi_\J$'s. This explains why in the sequel, we will use, for the half-plane representations, traces on lines which are contained in the interior of $\Omega$ instead of traces on the boundaries of the half-planes (see \cref{ssec: deriving HPR}).
	\label{rem:funct-details}
\end{remark}

\section{Fourier half-plane representation for a two-layered medium}
\label{sec: HPR two-layered}
In this section and the next one, we are concerned with the simplest case of layered media, composed of only two layers. We therefore consider a wavenumber function $\k$ given by
\begin{equation}
	\k(\x) := \left\{\begin{array}{ll}
		\km & \text{if } \x<0,\\
		\kp & \text{if } \x > 0,
	\end{array}\right.
	\label{def k dioptre}
\end{equation} 
where $\km$ and $\kp$ are positive constants. The next section provides the proof of \cref{Theorem: main} for a junction of such media. The present section focuses on the core ingredient of the proof, that is, the extension of the half-plane representation \cref{cas hom/ rep demi plan} to a two-layered medium. This requires us to introduce a \emph{generalized Fourier transform} which diagonalizes the $\x$-dependent part of the Helmholtz operator, more precisely the unbounded selfadjoint operator $A:\rmD(A)\subset L^2(\R)\longrightarrow L^2(\R)$ defined by
\begin{equation}
	\forall \Psi \in \rmD(A) := H^2(\R),\quad A\Psi:=-\partial^2_\x \Psi -\k^2\,\Psi.  
	\label{def: A}   
\end{equation}
We will see that this transform can be interpreted as an operator of ``decomposition'' on a family of \emph{generalized eigenfunctions} of $A$ which are introduced below.

\subsection{Generalized eigenfunctions}
As $A$ is selfadjoint, its spectrum is necessarily real. Its eigenelements are thus pairs $(\lambda,\Psi) \in \R \times H^2(\R)$ such that $A\Psi = \lambda\,\Psi$, that is,
\begin{equation}
	-\partial_\x^2\Psi -\k^2\,\Psi = \lambda\,\Psi 
	\quad \text{in }\R.
	\label{eq gen eigenfunctions}
\end{equation}
This equation can also be derived from the Helmholtz equation by separation of variables, i.e., by searching for solutions to $-\Delta u - \k^2\,u = 0$ in the form $u(\x,\y) = \Psi(\x)\,\e^{\pm\sqrt{\lambda}\y}$. The latter behavior in the $\y$-direction will be implicitly understood below in the physical interpretation of the generalized eigenfunctions.

It is easily seen that with the assumption $\Psi \in H^2(\R)$, the only possible solution to \cref{eq gen eigenfunctions} is $\Psi = 0$ for any $\lambda \in \R$, which shows that the point spectrum of $A$ is empty. However, if we forget this assumption and search for \emph{bounded} solutions, the dimension of the space of bounded solutions is 0 if $\lambda < -\max(\km^2,\kp^2)$, 1 if $-\max(\km^2,\kp^2) \leq \lambda < -\min(\km^2,\kp^2)$ and 2 if $\lambda \geq -\min(\km^2,\kp^2)$. We introduce a particular basis of this space. Denoting
\begin{equation}
	\Lambda^\pm := (-\kpm^2,+\infty),
	\label{eq:def Lambda_pm}
\end{equation}
we consider two functions $\Psi^\pm : \Lambda^\pm\times\R \longrightarrow \C$ defined for all $(\lambda,\x) \in \Lambda^\pm\times\R$ by
\begin{equation}
	\Psi^\pm(\lambda,\x) := \left\{\begin{array}{ll}
		\e^{\mp\i\beta^\pm(\lambda)\,\x} + R^\pm(\lambda)\,\e^{\pm\i\beta^\pm(\lambda)\,\x} &\text{if } \pm \x >0, \\[1mm]
		T^\pm(\lambda)\,\e^{\mp\i\beta^\mp(\lambda)\,\x} & \text{if }\mp \x>0,
	\end{array}\right.
	\label{eq: def fpg} 
\end{equation}
where
\begin{align}
	\beta^\pm(\lambda) &:= \left\{\begin{array}{ll}
		\sqrt{\lambda+\kpm^2} &\text{if } \lambda\geq-\kpm^2,\\[1mm]
		\i\sqrt{-\lambda-\kpm^2}&\text{if } \lambda<-\kpm^2,
	\end{array}\right. 
	\label{eq: def beta}
	\\[1mm]
	R^\pm(\lambda) &:= \frac{\beta^\pm(\lambda)-\beta^\mp(\lambda)}{\beta^+(\lambda)+\beta^-(\lambda)} \quad\text{and}\quad
	T^\pm(\lambda) := \frac{2\beta^\pm(\lambda)}{\beta^+(\lambda)+\beta^-(\lambda)}.
	\label{eq: coeff refl transm}
\end{align}

For a given $\lambda$, $\Psi^\pm(\lambda,\cdot)$ is called a \emph{generalized eigenfunction} of $A$, because it is a solution to the eigenvalue equation \cref{eq gen eigenfunctions}, but it does not belong to $\rmD(A)$ (nor to $L^2(\R)$), which justifies the word \emph{generalized}. This function can be interpreted as the response of the two-layered medium to an incident wave with unit amplitude and wavenumber $\beta^\pm(\lambda)$. In the notation $\Psi^\pm$, the sign indicates the side where the incident wave lies, i.e., $\x>0$ for $+$ and $\x<0$ for $-$. The behaviour of $\Psi^\pm(\lambda,\cdot)$ depends on whether $\beta^\pm(\lambda)$ is real or imaginary. To fix ideas, let us assume for instance that  $\km > \kp$, i.e., $-\km^2 < -\kp^2$.

First consider $\Psi^+(\lambda,\cdot)$ for a given $\lambda \in \Lambda^+ := (-\kp^2,+\infty)$. Noticing that both $\beta^-(\lambda)$ and $\beta^+(\lambda)$ are real, we see that function $\Psi^+(\lambda,\cdot)$ can be written more explicitly as  
\begin{equation*}
	\Psi^+(\lambda,\x) =  \left\{\begin{array}{ll}
		\displaystyle \e^{-\i\sqrt{\lambda+\kp^2}\,\x}
		+ R^+(\lambda)\,\e^{\i\sqrt{\lambda+\kp^2}\,\x} & \text{if } \x>0, \\[1mm]
		\displaystyle T^+(\lambda)\,\e^{-\i\sqrt{\lambda+\km^2}\,\x} & \text{if }\x<0.		
	\end{array} \right.
\end{equation*}
Hence, on the $+$ side, it represents the superposition of an incident wave which propagates from $+\infty$ and a reflected wave which propagates back towards $+\infty$ and whose amplitude is given by the \emph{reflection coefficient} $R^+(\lambda)$, which is real. On the $-$ side, it becomes a transmitted wave which propagates towards $-\infty$ and whose amplitude is given by the \emph{transmission coefficient} $T^+(\lambda)$, which is also real.

Consider now $\Psi^-(\lambda,\cdot)$ for a given $\lambda \in \Lambda^- := (-\km^2,+\infty)$. Two situations occur. On the one hand, if $\lambda > -\kp^2$, as both $\beta^-(\lambda)$ and $\beta^+(\lambda)$ are real, the interpretation of  $\Psi^-(\lambda,\cdot)$ is the same as above  by swapping the roles of both sides. On the other hand, if $\lambda \in (-\km^2,-\kp^2)$, then $\beta^-(\lambda)$ is still real, but $\beta^+(\lambda)$ becomes imaginary, so that
\begin{equation*}
	\Psi^-(\lambda,\x) =  \left\{\begin{array}{ll}
		\displaystyle \e^{\i\sqrt{\lambda+\km^2}\,\x}
		+ R^-(\lambda)\,\e^{-\i\sqrt{\lambda+\km^2}\,\x} & \text{if } \x<0, \\[1mm]
		\displaystyle T^-(\lambda)\,\e^{-\sqrt{-\lambda-\kp^2}\,\x} & \text{if }\x>0.		
	\end{array} \right.
\end{equation*}
Hence, on the $-$ side, it represents the superposition of an incident wave which propagates from $-\infty$ and a reflected wave with a reflection coefficient $R^-(\lambda)$ which is now a complex number such that $|R^-(\lambda)| = 1$. On the $+$ side, the transmitted wave is now evanescent. This case corresponds to the well-known phenomenon of total reflection.

These comments are summarized in \cref{fig: behavior fpg}. Of course, similar interpretations hold in the case where $\km < \kp$, swapping the roles of both sides.

\begin{figure}[t]
	\centering
	\begin{tabular}{[L{2.7cm}"M{4cm}"M{4cm}]}
		\thickcline{2-3}
		\multicolumn{1}{c"}{} &\vspace{.1cm}$\Psi^+(\lambda,\cdot)$ &\vspace{.1cm} $\Psi^-(\lambda,\cdot)$\\
		\thickhline
		\!$\lambda\in(-\kp^2,+\infty)$&\vspace{.2cm}
		{\centering\includegraphics{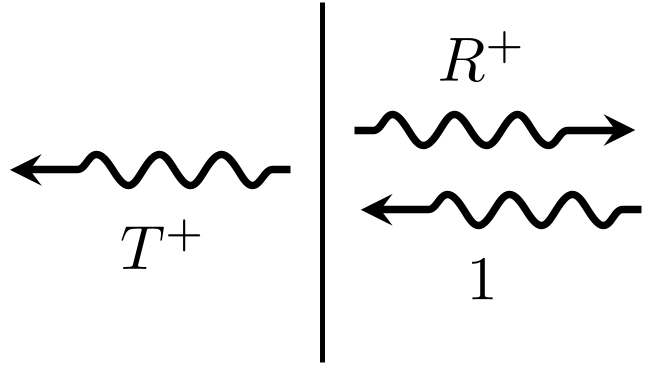}} &
		{\centering\includegraphics{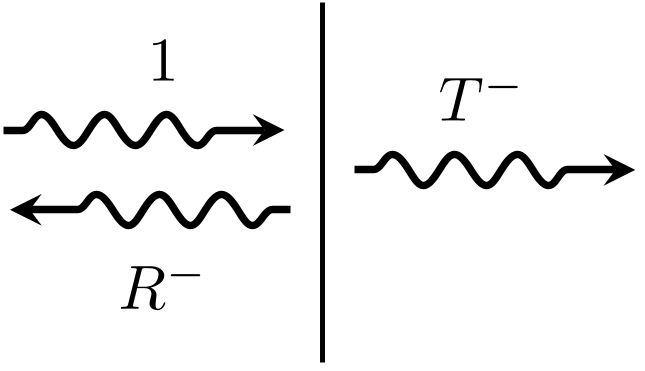}} 
		\\ \thickhline
		\!$\lambda\in(-\km^2,-\kp^2)$ &\vspace{.2cm} Undefined.
		&\vspace{.2cm}
		{\centering\includegraphics{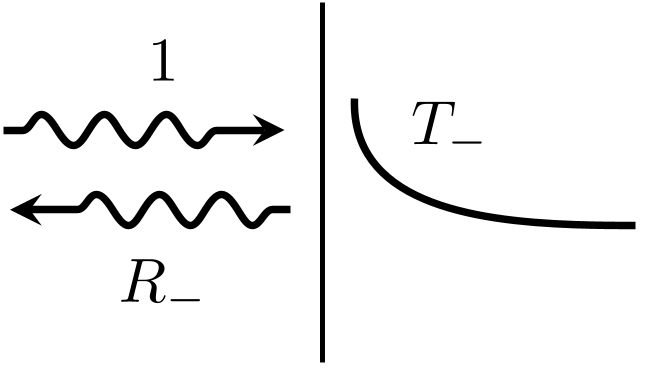}}  \\
		\thickhline
	\end{tabular}
	\caption{Behavior of the generalized eigenfunctions in the case where $\km > \kp$.}
	\label{fig: behavior fpg}
\end{figure}

The following proposition gathers some properties which are needed in the sequel.
\begin{proposition}
	{\rm (i)} The generalized eigenfunctions are bounded, in the sense that
	\begin{equation}
		\forall (\lambda,\x)\in\Lambda^\pm\times\R,\quad
		|\Psi^\pm(\lambda,\x)|\leq2.
		\label{bound Psi}
	\end{equation}
	
	{\rm (ii)} Consider \begin{equation}\label{eq def D}
		\mathbb{D} := \C \setminus (-\infty,-\min(\k_-^2,\k_+^2)].
	\end{equation} Each of the following functions has an analytic continuation in $\mathbb{D}$: $\beta^\pm(\lambda)$, $R^\pm(\lambda)$, $T^\pm(\lambda)$, as well as $\Psi^\pm(\lambda,\x)$  for any fixed $\x \in \R$. Moreover the analytic continuations of $\Psi^\pm(\lambda,\x)$ are bounded in any compact set of $\mathbb{D} \times \R$.
	\label{prop: fpg}
\end{proposition}
\begin{proof}
{\rm (i)} As $\beta^\pm(\lambda)\in\R^+\cup\i\R^+$ for $\lambda\in\Lambda^\pm$, we have $|R^\pm(\lambda)|\leq 1$ and $|T^\pm(\lambda)|\leq 2$ on $\Lambda^\pm$, which yields \cref{bound Psi} since the exponentials involved in the expression \cref{eq: def fpg} are either oscillating or evanescent.

{\rm (ii)} Consider the definition \cref{eq: def beta} of $\beta^\pm(\lambda)$ restricted to $\lambda \in (-\k_\pm^2,+\infty)$. Using the principal determination of the complex square root, i.e., $\sqrt{z} := |z|^{1/2}\e^{\i(\arg z) / 2}$ with $\arg z \in (-\pi,+\pi)$, we see that  $\beta^\pm(\lambda)$ has an analytic continuation in $\C \setminus (-\infty,-\k_\pm^2]$. Hence, as both functions $\beta^-(\lambda)$ and $\beta^+(\lambda)$ are involved in the definition \cref{eq: coeff refl transm} of $R^\pm(\lambda)$ and $T^\pm(\lambda)$, the latter functions extend to analytic functions in $\mathbb{D}$, which is the common domain of analyticity of $\beta^-(\lambda)$ and $\beta^+(\lambda)$ (note that the denominator $\beta^+(\lambda) + \beta^-(\lambda)$ cannot vanish in $\mathbb{D}$). Therefore the same holds for $\Psi^\pm(\lambda,\x)$ for any fixed $\x \in \R$. Moreover, as $\e^{\pm\i\beta^\pm(\lambda)\,\x}$ is bounded in any compact set of $\mathbb{D} \times \R$, so is $\Psi^\pm(\lambda,\x)$. \end{proof}
\begin{remark} In this section we interpreted complex exponentials as being ingoing or outgoing based on an implicit convention
, a time-dependence in $\e^{-\i\omega t}$. For a time-dependence in $\e^{\i\omega t}$, the $\Psi^\pm$'s would then be replaced by their conjugates and everything that follows would hold true with this alternative choice.
	\end{remark}

\subsection{Generalized Fourier transform}
The usual Fourier transform $F$ defined in \cref{def: usual Fourier} can be interpreted as an operator of ``decomposition'' on the family of functions $\Psi_\xi(\x) := \e^{\i\xi\x}$ for $\xi\in\R$, which appear as generalized eigenfunctions of the operator $-\i \partial_\x$ in the sense that they satisfy $-\i \partial_\x \Psi_\xi = \xi\,\Psi_\xi$. Moreover $F$ diagonalizes $-\i \partial_\x$, in that it transforms the action of this operator into a multiplication by $\xi$, which can be written in a synthetic form $-\i \partial_\x = F^{-1} \xi \,F$. In this relation, $\xi$ plays the role of a spectral variable that lies in the spectrum of $-\i \partial_\x$, which covers the whole real line $\R$.

Similar properties hold for every selfadjoint operator in a Hilbert space (see, e.g. \cite{Berezansky-etal-1996}). \Cref{th Fourier gen} below tells us that that the projection on the family of generalized eigenfunctions defined in \cref{eq: def fpg} diagonalizes the operator $A$ introduced in \cref{def: A}. In order to make precise the functional framework in which this diagonalization takes place, we need to introduce a few notations to complement \cref{eq:def Lambda_pm}--\cref{eq: coeff refl transm}.

Considering the functions $\rho^\pm: \Lambda^\pm \to \R$ given by
\begin{equation}
	\forall \lambda\in\Lambda^\pm,\quad 
	\rho^\pm(\lambda) := \frac{1}{4\pi\,\beta^\pm(\lambda)} 
	= \frac{1}{4\pi\sqrt{\lambda+\kpm^2}},
	\label{eq: def rho}
\end{equation}
we denote $L^2(\Lambda^\pm,\rho^\pm )$ the weighted space of square integrable functions on $\Lambda^\pm$ for the measure $\rho^\pm\d\lambda$, i.e., the space composed of measurable functions $\Hphi^\pm : \Lambda^\pm \to \C$ such that
\begin{equation*}
	\big\| \Hphi^\pm \big\|^2_{L^2(\Lambda^\pm,\rho^\pm )}
	:= \int_{\Lambda^\pm} \big|\Hphi^\pm(\lambda)\big|^2 
	\,\rho^\pm(\lambda)\,\d\lambda <+\infty.
\end{equation*}
We will also use the space $L^1(\Lambda^\pm,\rho^\pm )$ composed of measurable functions $\Hphi^\pm : \Lambda^\pm \to \C$ such that
\begin{equation*}
	\int_{\Lambda^\pm} \big|\Hphi^\pm(\lambda)\big|
	\,\rho^\pm(\lambda)\,\d\lambda <+\infty.
\end{equation*}
Then we define the Hilbert space
\begin{equation*}
	\HH := L^2(\Lambda^-,\rho^- ) \times L^2(\Lambda^+,\rho^+),
\end{equation*}
which acts as the spectral space in the diagonalization theorem below. The proof of this theorem is based on the spectral theorem, following the same lines as in the pioneering book of Titchmarsh \cite{Titchmarsh-1946}. An explicit construction can be found for instance in \cite{Ott-2017}.

\begin{theorem}\label{th Fourier gen}
	The generalized Fourier transform $\calF$ defined for all $\varphi\in L^1(\R) \cap L^2(\R)$ by
	\begin{equation}
		\calF\varphi := (\calF^-\varphi,\calF^+\varphi) 
		\quad\text{with}\quad
		\calF^\pm\varphi(\lambda) := \int_\R\varphi(\x)\,\overline{\Psi^\pm(\lambda,\x)}\,\d\x \quad\text{for }\lambda\in\Lambda^\pm,
		\label{eq : thm gft f}
	\end{equation}
	extends by density to a unitary operator from $L^2(\R)$ to $\HH$. It diagonalizes $A$ in the sense that
	\begin{equation}
		\forall \varphi\in H^2(\R),\quad  A\varphi=\calF^{-1} \lambda\,\calF\varphi, 
		\label{eq: diagonalisation}
	\end{equation}
	where the notation $\lambda$ stands for the operator of multiplication by $\lambda$ in $\HH$. Moreover, the inverse transform $\calF^{-1}$ writes explicitly as follows, for all $(\Hphi^-,\Hphi^+)\in\HH$ such that $\Hphi^\pm \in L^1(\Lambda^\pm,\rho^\pm )$:
	\begin{equation}
		\forall \x\in\R, \quad
		\calF^{-1} (\Hphi^-,\Hphi^+)(\x) = \sum_{\pm}\int_{\Lambda^\pm}
		\Hphi^\pm(\lambda)\,\Psi^\pm(\lambda,\x)\,\rho^\pm(\lambda)\,\d\lambda.
		\label{eq : thm gft f-1}
	\end{equation}
\end{theorem}

\begin{remark}
	Note that formula \cref{eq : thm gft f} makes sense for all functions $\varphi \in L^1(\R)$, since \cref{bound Psi} tells us that $\Psi^\pm(\lambda,\x)$ remains bounded. But for $\varphi \in L^2(\R) \setminus L^1(\R)$, it becomes an improper integral whose definition follows from a density argument as in the definition of the usual Fourier transform in $L^2(\R)$. Again thanks to \cref{bound Psi}, the same holds true for formula \cref{eq : thm gft f-1}, which makes sense for all $\Hphi^\pm \in L^1(\Lambda^\pm,\rho^\pm )$.
	\label{rem Fourier gen L2-L1}
\end{remark}

\Cref{th Fourier gen} provides us a spectral decomposition of $L^2(\R)$ associated with operator $A$. On the one hand, $\calF$ acts as an operator of ``decomposition'' on the family composed of the $\Psi^\pm$'s, in the sense that it furnishes the spectral components $\calF\varphi$ of $\varphi$. These components are functions of the variable $\lambda$ that lies in the spectrum of $A$, which is purely continuous and given by $\overline{\Lambda^- \cup \Lambda^+} = [-\max(\km^2,\kp^2),+\infty)$. On the other hand, $\calF^{-1}$ appears as an operator of ``recomposition'' which allows us to reconstruct a function $\varphi$ from its spectral components.

It is instructive to understand how this generalized Fourier transform relates to the usual one in the particular case of a homogeneous medium, that is, when $\km=\kp$. Simply denoting $\k$ this common value, we have in this case $\Lambda^- = \Lambda^+ = (-\k^2,+\infty)$ and
\begin{equation*}
	\forall \lambda\in (-\k^2,+\infty),\quad 
	\beta^-(\lambda) =\beta^-(\lambda)= \sqrt{\lambda+\k^2}, \quad
	R^\pm(\lambda) = 0 \quad\text{and}\quad
	T^\pm(\lambda) = 1.
\end{equation*}
Therefore,
\begin{equation*}
	\forall (\lambda,\x) \in (-\k^2,+\infty) \times \R, \quad 
	\Psi^\pm(\lambda,\x) = \e^{\mp\i \sqrt{\lambda+\k^2}\,\x}.
\end{equation*}
This shows that the definition \cref{eq : thm gft f} of the generalized Fourier transform simplifies as
\begin{equation*}
	\calF\varphi := (\calF^-\varphi,\calF^+\varphi) 
	\quad\text{with}\quad
	\calF^\pm\varphi(\lambda) := \int_\R\varphi(\x)\,\e^{\pm\i \sqrt{\lambda+\k^2}\,\x}\,\d\x \quad\text{for }\lambda\in(-\k^2,+\infty).
\end{equation*}
We observe that we return to the definition \cref{def: usual Fourier} of the usual Fourier transform simply by setting $\xi = \mp \sqrt{\lambda+\k^2}$. More precisely, introducing the operator $\calC$ associated to this change of variable, defined for a function $g$ of the $\xi$-variable by
\begin{equation*}
	\calC g := \big( \calC^- g \,\,, \calC^+ g\big)
	\quad\text{where}\quad
	\calC^\pm g\,(\lambda) := g\big(\mp\sqrt{\lambda+\k^2}\big)
	\quad\text{for }\lambda\in(-\k^2,+\infty),
\end{equation*}
we have
\begin{equation*}
	\calF = \calC\,F.
\end{equation*}
Note that this formula allows us to derive the unitary nature of $\calF$ asserted in \cref{th Fourier gen} from that of the usual Fourier transform. Indeed we know on the one hand that $F$ is unitary from $L^2(\R_\x)$ to $L^2(\R_\xi, (2\pi)^{-1})$. On the other hand, noticing that
\begin{equation*}
	\frac{1}{2\pi} \int_{\R} \big| g(\xi) \big|^2\,\d\xi 
	= \sum_{\pm} \int_{-\k^2}^{+\infty} \big| g(\mp\sqrt{\lambda+\k^2}) \big|^2\,\frac{1}{4\pi\sqrt{\lambda+\k^2}}\,\d\lambda,
\end{equation*}
we infer that $\calC$ is unitary from $L^2(\R_\xi, (2\pi)^{-1})$ to $\HH$. Thus their product is unitary from $L^2(\R_\x)$ to $\HH$. The above formula explains the presence of the weight functions $\rho^\pm(\lambda)$ in the spectral space of the generalized Fourier transform.

\subsection{Deriving the half-plane representation}
\label{ssec: deriving HPR}
We are now able to generalize the half-plane representation \cref{cas hom/ rep demi plan} to a two-layered medium. The technical difficulty mentioned in \cref{rem:funct-details} leads us to modify our definition \cref{def half plane} of $\rmH$, without loss of generality. From now on, $\rmH$ denotes the half-plane defined in a local coordinate system $(\x,\y)$ by
\begin{equation}
	\rmH := \{ (\x,\y) \in \R^2 \mid \y > -\varepsilon \},
	\label{def half plane eps}
\end{equation}
where $\varepsilon > 0$ is a fixed parameter whose role is simply that the $\x$-axis $\Sigma := \R \times \{0\}$ is now contained in $\rmH$. Our aim is to derive a half-plane representation of a solution $u\in L^2(\rmH)$ to the Helmholtz equation in $\rmH$, i.e., to express $u$ by means of its trace on $\Sigma$. As $\Sigma \subset \rmH$, interior regularity of solutions to the Helmholtz equation \cref{eq: Helmholtz in H} shows this trace is well defined as a continuous function (since $u \in H^2_{\rm loc}(\rmH)$). 

\begin{proposition} \label{prop: HPR}
	Suppose that $u\in L^2(\rmH)$ satisfies 
	\begin{equation}
		-\Delta u - \k^2 u = 0  \quad\text{in } \rmH,
		\label{eq: Helmholtz in H}
	\end{equation}
	in the distributional sense, where $\k = \k(\x)$ is given in \cref{def k dioptre}. Then,
	\begin{equation}
		\forall (\x,\y) \in \R \times [0,+\infty),\quad 
		u(\x,\y) = \sum_\pm \int_{\R^+} \Hphi^\pm(\lambda)\,
		\Psi^\pm(\lambda,\x) \,\e^{-\sqrt{\lambda}\y}\,\rho^\pm(\lambda)\,\d\lambda,
		\label{eq: HPR}
	\end{equation}
	where $\Psi^\pm$ and $\rho^\pm$ are defined respectively in \cref{eq: def fpg} and \cref{eq: def rho}, and $\Hphi^\pm := \calF^\pm u(\cdot,0)$ has the following properties:
	\begin{equation}
		\Hphi^\pm(\lambda)= 0 \text{ if }\lambda \in(-\kpm^2,0)
		\quad\text{and}\quad
		\lambda^s\,\Hphi^\pm(\lambda) \in L^1(\R^+,\rho^\pm) \text{ for } s > -\frac{3}{4}.
		\label{eq: HPR prop phi}
	\end{equation}
\end{proposition}

\begin{remark}
	Note that the integral in \cref{eq: HPR} is well defined (in the Lebesgue sense). Indeed $\Psi^\pm(\lambda,\x)$ is bounded (see \cref{bound Psi}), as well as $\e^{-\sqrt{\lambda}\y}$ for $\y \in [0,+\infty)$, and $\Hphi^\pm$ belongs to $L^1(\Lambda^\pm,\rho^\pm)$ (according to \cref{eq: HPR prop phi} with $s=0$).
	\label{rem: integ HPR}
\end{remark}

\begin{proof}
	As $u\in L^2(\rmH)$, we know that for almost every $\y\in (-\varepsilon,+\infty)$, we have $u(\cdot,\y)\in L^2(\R)$, so that we can define its partial generalized Fourier transform in the $\x$-direction by 
	$\Hu(\cdot\,,\y) 
	:= \big( \Hu^-(\cdot\,,\y),\Hu^+(\cdot\,,\y)\big)$ 
	where 
	$\Hu^\pm(\cdot\,,\y) := \calF^\pm u(\cdot,\y).$ As $\calF$ is  unitary from $L^2(\R)$ to $\HH$, we have
	\begin{equation*}
		\| u(\cdot\, ,\y) \|_{L^2(\R)}^2 
		= \sum_\pm \big\| \Hu^\pm(\cdot\, ,\y) \big\|_{L^2(\Lambda^\pm,\rho^\pm)}^2.
	\end{equation*}
	Integrating on $(-\varepsilon,+\infty)$ with respect to $\y$ then yields
	\begin{align}
		\| u \|_{L^2(\rmH)}^2 
		& = \sum_\pm \int_{-\varepsilon}^{+\infty} 
		\int_{\Lambda^\pm} \big| \Hu^\pm(\lambda,\y) \big|^2
		\,\rho^\pm(\lambda)\,\d\lambda\,\d\y \label{eq: Parseval u} \\
		& = \sum_\pm \int_{\Lambda^\pm} \big\| \Hu^\pm(\lambda,\cdot) \big\|_{L^2(-\varepsilon,+\infty)}^2
		\,\rho^\pm(\lambda)\,\d\lambda, \nonumber
	\end{align}
	where the second equality follows from Fubini's theorem. This shows in particular that
	\begin{equation}
		\text{for almost every }\lambda \in \Lambda^\pm,\quad
		\Hu^\pm(\lambda,\cdot) \in L^2(-\varepsilon,+\infty).
		\label{eq: partial GFT L2}
	\end{equation}
	
	We prove now that applying the partial generalized Fourier transform to the Helmholtz equation \cref{eq: Helmholtz in H} yields 
	\begin{equation}
		- \partial^2_\y \Hu^\pm(\lambda,\cdot) 
		+ \lambda\,\Hu^\pm(\lambda,\cdot) = 0 
		\quad \text{in }(-\varepsilon,+\infty)
		\label{eq:Fourier Helmholtz}
	\end{equation}
	in the distributional sense, for almost every $\lambda\in\Lambda^\pm$. The idea is to use test functions with separated variables, i.e., in the form $\psi(\x)\,\chi(\y)$. Indeed, the Helmholtz equation \cref{eq: Helmholtz in H} amounts to saying that for all $\psi \in \calD(\R)$ and $\chi \in \calD(-\varepsilon,+\infty)$ (usual spaces of compactly supported indefinitely differentiable functions), we have
	\begin{equation*}
		\big\langle \big\langle -\partial^2_\x u - \partial^2_\y u - \k^2 u, \psi\,\chi \big\rangle_{\!\x} \, \big\rangle_{\!\y} = 0,
	\end{equation*}
	where we distinguish the anti-duality products in the $\x$ and $\y$ variables, which can be interchanged. From the definition of derivatives in the distributional sense, we deduce that
	\begin{equation*}
		\big\langle \big\langle u\,, -(\partial^2_\x \psi)\,\chi- \psi\,(\partial^2_\y\chi) \big\rangle_{\!\x} \, \big\rangle_{\!\y}
		 -\big\langle \big\langle \k^2 u\,, \psi\,\chi \big\rangle_{\!\x} \, \big\rangle_{\!\y} = 0.
	\end{equation*}
	or equivalently
	\begin{equation*}
		-\Big\langle \big\langle u\,,\chi \big\rangle_{\!\y}\,, \partial^2_\x\psi \Big\rangle_{\!\x}
		-\Big\langle \k^2\, \big\langle u\,,\chi \big\rangle_{\!\y}\,, \psi \Big\rangle_{\!\x}
		- \Big\langle \big\langle u\,,\partial^2_\y\chi \big\rangle_{\!\y}\,, \psi \Big\rangle_{\!\x} = 0.
	\end{equation*}
	As $u \in L^2(\rmH)$, the anti-duality products in this expression can be replaced by integrals, which allows us to rewrite it as
	\begin{equation*}
		\int_\R \big\langle u(\x,\cdot),\chi\big\rangle_{\!\y}\,
		\overline{\big( -\partial^2_\x \psi(\x) - \k(\x)^2 \psi(\x)  \big)}\,\d\x
		- \int_\R \big\langle u(\x,\cdot),\partial^2_\y\chi\big\rangle_{\!\y}\, \overline{\psi(\x)}
		\,\d\x = 0.
	\end{equation*}
	Interpreting both integrals as inner products in $L^2(\R)$, we can use the generalized Plancherel's identity following from the unitary nature of $\calF$ from $L^2(\R)$ to $\HH$. Denoting $(\Hpsi^-,\Hpsi^+) := \calF \psi$, we infer that the first integral is equal to
	\begin{equation*}
		\sum_{\pm}\int_{\Lambda^\pm}  
		\big\langle \Hu^\pm(\lambda,\cdot),\chi\big\rangle_{\!\y} 
		\ \overline{\lambda \, \Hpsi^\pm(\lambda)}\,\rho^\pm(\lambda)\,\d\lambda.
	\end{equation*}
	Indeed, on the one hand, we have $\calF(-\partial^2_\x \psi - \k^2 \psi) = \calF(A \psi) = \lambda\,\calF \psi$ thanks to \cref{eq: diagonalisation}. On the other hand, the generalized Fourier transform of $\x \to \big\langle u(\x,\cdot),\chi\big\rangle_{\!\y}$ is given by $\big\langle \Hu^\pm(\lambda,\cdot),\chi\big\rangle_{\!\y}$, which amounts to interchanging $\calF$ with the $\y$-anti-duality products. This permutation is obvious when $\calF$ applies to functions with separated variables; as these functions span $L^2(\rmH)$, the general case follows from a density argument. Computing the second integral yields finally
	\begin{equation*}
		\sum_{\pm}\int_{\Lambda^\pm}  
		\bigg( \lambda \, \big\langle \Hu^\pm(\lambda,\cdot),\chi\big\rangle_{\!\y}
		- \big\langle \Hu^\pm(\lambda,\cdot),\partial^2_\y\chi\big\rangle_{\!\y} \bigg)
		\,\overline{\Hpsi^\pm(\lambda)}\,\rho^\pm(\lambda)\,\d\lambda = 0,
	\end{equation*}
	As this equality holds true for all $\psi \in \calD(\R)$, we infer that the quantity between parentheses must vanish for almost every $\lambda \in \Lambda^\pm$ and every $\chi \in \calD(-\varepsilon,+\infty)$, which is exactly equation \cref{eq:Fourier Helmholtz}.
	
	From this equation, we deduce that there exist functions $A^\pm(\lambda)$ and $B^\pm(\lambda)$ such that
	\begin{equation*}
		\Hu^\pm(\lambda,\y) = A^\pm(\lambda) \,\e^{-\sqrt{\lambda}(\y+\varepsilon)} + B^\pm(\lambda) \,\e^{+\sqrt{\lambda}(\y+\varepsilon)},
	\end{equation*}
	where we keep the convention $\sqrt{\lambda} := \i\sqrt{-\lambda}$ if $\lambda<0$. Thanks to \cref{eq: partial GFT L2}, we infer that 
	$A^\pm(\lambda)=0$ if $\lambda\in(-\kpm^2,0)$ and $B^\pm(\lambda) = 0$ for all $\lambda\in\Lambda^\pm$. Therefore, 
	\begin{equation*}
		\Hu^\pm(\lambda,\y) 
		= A^\pm(\lambda)\,\e^{-\sqrt{\lambda}(\y+\varepsilon)} 
		\quad\text{where}\quad 
		A^\pm (\lambda) = 0\text{ if }\lambda\in(-\kpm^2,0).
	\end{equation*}
	According to the definition $\Hphi^\pm := \calF^\pm u(\cdot,0)$, this can be written equivalently
	\begin{equation*}
		\Hu^\pm(\lambda,\y) 
		= \Hphi^\pm(\lambda)\,\e^{-\sqrt{\lambda}\y} 
		\quad\text{where}\quad 
		\Hphi^\pm(\lambda) = 0\text{ if }\lambda\in(-\kpm^2,0).
	\end{equation*}
	Using the expression \cref{eq : thm gft f-1} of $\calF^{-1}$, we obtain the half-plane representation \cref{eq: HPR}, where the existence of the integral will be justified, once \cref{eq: HPR prop phi} is proven (see \cref{rem: integ HPR}).
	
	To do so, notice that for almost every $\lambda\in\Lambda^\pm$,
	\begin{equation*}
		\big\| \Hu^\pm(\lambda,\cdot) \big\|^2_{L^2(-\varepsilon,+\infty)} 
		= \big|\Hphi^\pm(\lambda)\big|^2 \, 
		\int_{-\varepsilon}^{+\infty} \e^{-2\sqrt{\lambda}\y}\,\d\y
		= \frac{\big|\Hphi^\pm(\lambda)\big|^2
			 \,\e^{2\sqrt{\lambda}\varepsilon}}{2\sqrt{\lambda}}.
	\end{equation*}
	Hence we deduce from \cref{eq: Parseval u} that
	\begin{equation*}
		\| u \|_{L^2(\rmH)}^2 
		= \sum_\pm \int_{\R^+} \frac{\big|\Hphi^\pm(\lambda)\big|^2 \,\e^{2\sqrt{\lambda}\varepsilon}}{2\sqrt{\lambda}}
		\,\rho^\pm(\lambda)\,\d\lambda,
	\end{equation*}
	which shows that $\lambda^{-1/4}\,\e^{\sqrt{\lambda}\varepsilon}\,\Hphi^\pm(\lambda) 
	\in L^2(\R^+,\rho^\pm).$ Therefore Cauchy-Schwartz inequality yields
	\begin{equation*}
		\int_{\R^+} \lambda^s\, \big|\Hphi^\pm(\lambda)\big| \rho^\pm(\lambda)\,\d\lambda
		\leq 
		\left\|
		\lambda^{-1/4}\,\e^{\sqrt{\lambda}\varepsilon}\,\Hphi^\pm(\lambda)
		\right\|_{L^2(\R^+,\rho^\pm)} \
		\left\|
		\lambda^{s+1/4}\,\e^{-\sqrt{\lambda}\varepsilon}
		\right\|_{L^2(\R^+,\rho^\pm)},
	\end{equation*}
	where the last norm is bounded provided $s>-3/4$, which proves \cref{eq: HPR prop phi}.
\end{proof}

The following corollary of \cref{prop: HPR} tells us about the behaviour of $u$ in any non-transverse direction.

\begin{corollary} \label{cor: HPR}
	Let $u$ defined as in \cref{prop: HPR}. Then for any half-line $\Gamma_\alpha:=\{(t\cos\alpha,  t\sin\alpha) \in\R^2 \mid t>0\}\subset \rmH$ where $\alpha\in(0,\pi)$, we have $u_{|\Gamma_\alpha}\in L^1(\Gamma_\alpha)$.  
\end{corollary}

\begin{proof}
	Using the half-plane representation \cref{eq: HPR}, then \cref{bound Psi} and Fubini's theorem, we obtain	
	\begin{align*}
		&\int_0^{+\infty}|u(t\cos\alpha,t\sin\alpha)|\,\d t \\
		&\hspace{2cm}= \int_0^{+\infty}\left| \sum_\pm \int_{\R^+} \Hphi^\pm(\lambda)\,
		\Psi^\pm(\lambda,t\cos\alpha) \,\e^{-\sqrt{\lambda}t\sin\alpha}
		\,\rho^\pm(\lambda)\,\d\lambda  \right| \d t \\
		&\hspace{2cm}\leq 2 \sum_{\pm}\int_0^{+\infty} \int_{\R^+} 
		\left| \Hphi^\pm(\lambda) \right|\,
		\e^{-\sqrt{\lambda}t\sin\alpha} \,\rho^\pm(\lambda)\,\d\lambda \d t \\
		&\hspace{2cm}\leq 2 \sum_\pm \int_{\R^+} 
		\frac{\left| \Hphi^\pm(\lambda) \right|}{\sqrt{\lambda}\sin\alpha}
		\,\rho^\pm(\lambda)\,\d\lambda,
	\end{align*}
	where the last integral is bounded, according to \cref{eq: HPR prop phi} with $s = -1/2$.
\end{proof}

\section{Junction of two-layered media}
\label{sec: two-layer}
In this section, we come back to the junction of stratified media as defined in \cref{setting pb}, but we restrict ourselves to the particular case of two-layered media, as considered in \cref{sec: HPR two-layered}. Hence each of the three half-planes $\rmH_\W$, $\rmH_\N$ and $\rmH_E$ is characterized by a two-valued wavenumber function $\k_\J$ for $\J \in \{\W,\N,\E\}$, where $\k_\J(\x) = \k_{\J,\pm}$ if $\pm\x_\J > 0$. By construction, we have
\begin{multline}
	\k_{\W,+} = \k_{\N,-} \quad \text{and} \quad \k_{\E,-} = \k_{\N,+},\\
	\text{as well as}\quad \k_{\W,-} = \k_{\E,+} 
	\quad\text{if}\quad (\theta_\W,\theta_\E ) \neq (+\pi/2,-\pi/2).
	\label{eq: k two-layer}
\end{multline}

The various objects introduced in \cref{sec: HPR two-layered} for a single half-plane will simply be indexed by $\J\in\{\W,\N,\E\}$: the transverse wavenumbers $\beta^\pm_\J(\lambda)$, the reflection and transmission coefficients $R^\pm_\J(\lambda)$ and  $T^\pm_\J(\lambda)$, the generalized eigenfunctions $\Psi^\pm_\J(\lambda,\x_\J)$, the weight functions $\rho^\pm_\J(\lambda)$ and the generalized Fourier transforms $\calF_\J$ (see \cref{eq: def fpg}--\cref{eq: coeff refl transm}, \cref{eq: def rho} and \cref{th Fourier gen}).
 
Recall that we have replaced the initial definition \cref{def half plane} of $\rmH$ by \cref{def half plane eps} so that the $\x$-axis $\Sigma := \R \times \{0\}$ is now contained in $\rmH$. The same goes for each half-plane $\rmH_\J$ for $\J \in \{\W,\N,\E\}$, whose local coordinates are still related to the global ones $(x,y)$ by \cref{eq: coord locales}. We denote by $\Sigma_\J$ the $\x_\J$-axis of $\rmH_\J$, which is contained in $\rmH_\J$. Note that the three points $(a_\W,b_\W)$, $(0,0)$ and $(a_\E,b_\E)$ are located respectively at the intersection of $\Sigma_\J$ with the interface of the two-layered medium filling $\rmH_\J$.

We prove now \cref{Theorem: main} in the context described above. For the sake of clarity, we begin by the case of the right-angle configuration (see \cref{fig: cas non hom/4k}, left).

\begin{figure}[t]
		\centering
		\includegraphics[height=5cm]{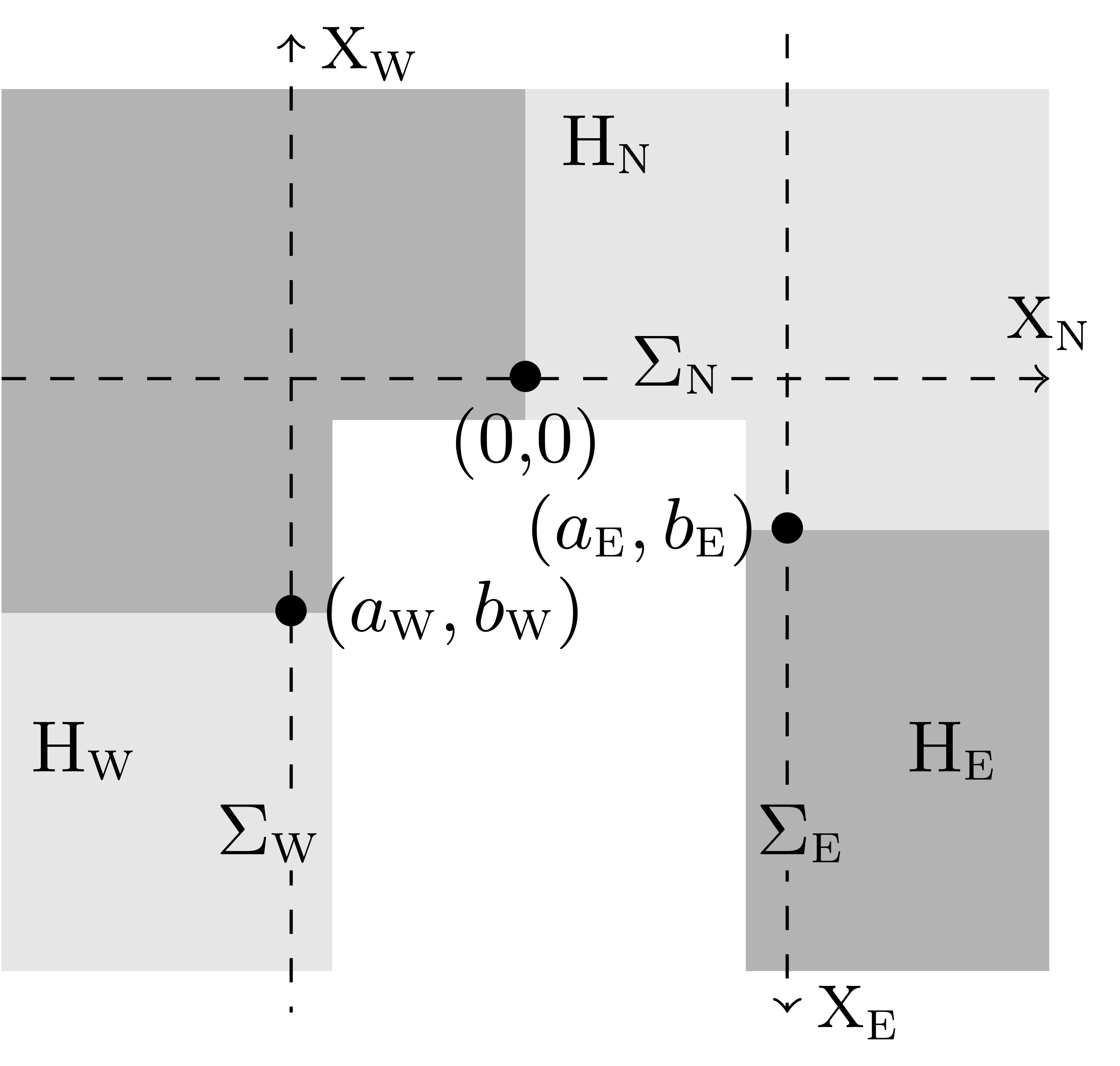}
	\hspace{.05\columnwidth}
	\includegraphics[height=5cm]{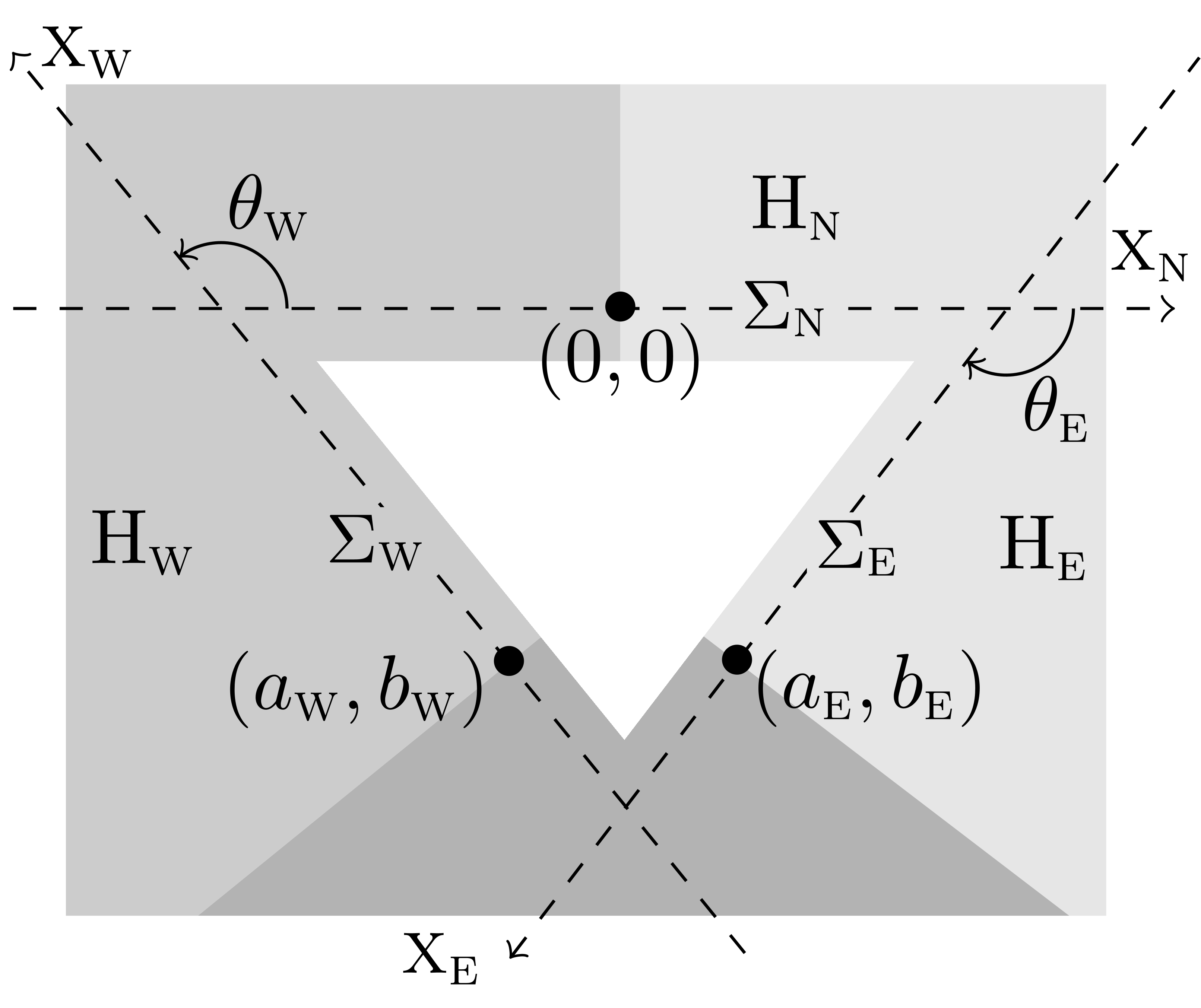}
\caption{Junctions of two-layered media.}\label{fig: cas non hom/4k}
\end{figure}
\subsection{Right-angle configuration} 
\label{subsection: right angle}
We follow the same approach as in the second step of \cref{section : homogeneous case}.
Recall that in the right-angle configuration, the relations between the local and global coordinate systems are given by \cref{eq: coord locales right angle}.

Suppose that $u \in L^2(\Omega)$ satisfies the Helmholtz equation in $\Omega := \rmH_\W \cup \rmH_\N \cup \rmH_\E$. Hence, we can use the half-plane representation of \cref{prop: HPR} for the restriction of $u$ in each of the three half-planes. For $\J \in \{\W,\N,\E\}$, we denote by  $\varphi_\J$ the trace of $u$ on $\Sigma_\J$, considered as a function of the local coordinate $\x_\J$, and by $(\Hphi_\J^-,\Hphi_\J^+) := \calF_\J \varphi_\J$ its generalized Fourier transform with respect to $\x_\J$. 

In order to prove that $u=0$, the key argument consists in proving that $\Hphi_\N^\pm(\lambda)$ extends to an analytic function of $\lambda$ in the domain
\begin{equation}
	\mathbb{D}_\N := \C \setminus (-\infty,-\min(\k_{\N,-}^2,\k_{\N,+}^2)]
	\label{eq:def-DN}
\end{equation}
introduced in \cref{prop: fpg}. As \cref{eq: HPR prop phi} tells us that $\Hphi_\N^\pm(\lambda)$ vanishes on the interval $(-\k_{\N,\pm}^2,0)$, analyticity implies that it vanishes also on $(0,+\infty)$, hence $\Hphi_\N^\pm = 0$. The half-plane representation \cref{eq: HPR} in $\rmH_\N$ then tells us that $u$ vanishes in $\rmH_\N$, so finally also in the whole domain $\Omega$ by virtue of the unique continuation principle.

It remains to prove that $\Hphi_\N^\pm(\lambda)$ has an analytic continuation in $\mathbb{D}_\N$. We focus on the case of $\Hphi_\N^+(\lambda)$. The case of $\Hphi_\N^-(\lambda)$ can be dealt with in the same way. We start from the definition \cref{eq : thm gft f} of  $\Hphi_\N^+(\lambda)$ (recall that $\x_\N=x$):
\begin{equation*}
\forall \lambda \in (-\k_{\N,+}^2,+\infty), \quad
\Hphi_\N^+(\lambda):=\int_\R\varphi_\N(x)\,\overline{\Psi_\N^+(\lambda,x)}\,\d x.
\end{equation*}
According to \cref{rem Fourier gen L2-L1}, this integral is well defined (in the Lebesgue sense) if $\varphi_\N \in L^1(\R)$, which is not obvious. This actually follows \cref{cor: HPR} applied for both half-planes $\rmH_\W$ and $\rmH_\E$, which ensures that $u|_{\Sigma_\N \cap \rmH_\W}$ and $u|_{\Sigma_\N \cap \rmH_\E}$ are integrable, that is, $\varphi_\N|_{(-\infty,a_\W)} \in L^1(-\infty,a_\W)$ and $\varphi_\N|_{(a_\E,+\infty)} \in L^1(a_\E,+\infty)$. We can then split the above integral in three parts as follows:
\begin{multline}
	\Hphi_\N^+(\lambda)= \\
	\underbrace{\int_{-\infty}^{a_\W}\varphi_\N(x)\,\overline{\Psi_\N^+(\lambda,x)}\,\d x}_{ \displaystyle =:\Hphi_{\N,\W}^+(\lambda)}
	+ \underbrace{\int_{a_\W}^{a_\E}\varphi_\N(x)\,\overline{\Psi_\N^+(\lambda,x)}\,\d x}_{ \displaystyle =:\Hphi_{\N,0}^+(\lambda)}
	+ \underbrace{\int_{a_\E}^{+\infty}\varphi_\N(x)\,\overline{\Psi_\N^+(\lambda,x)}\,\d x}_{ \displaystyle =:\Hphi_{\N,\E}^+(\lambda)}.
	\label{eq:split TphiN}
\end{multline}

First consider the second integral $\Hphi_{\N,0}^+(\lambda)$. Recall that \Cref{prop: fpg} tells us that for any $x \in \R$, the function $(-\k_{\N,+}^2,+\infty)\ni\lambda\to{\Psi_\N^+}(\lambda,x)$ has an analytic continuation in $\mathbb{D}_\N$ and $(\lambda,x)\to{\Psi_\N^+}(\lambda,x)$  is bounded in any compact set of $\mathbb{D}_\N \times \R$. Then, the function  $(-\k_{\N,+}^2,+\infty)\ni\lambda\to\overline{\Psi_\N^+(\lambda,x)}$ also has an analytic continuation in $\mathbb{D}_\N$, simply defined by $\overline{\Psi_\N^+(\overline\lambda,x)}$, and $(\lambda,x)\to\overline{\Psi_\N^+(\overline\lambda,x)}$  is bounded in any compact set of $\mathbb{D}_\N \times \R$. Using Morera's theorem, we infer that $\Hphi_{\N,0}^+(\lambda)$ has an analytic continuation in $\mathbb{D}_\N$.

These arguments cannot apply for $\Hphi_{\N,\W}^+(\lambda)$ nor $\Hphi_{\N,\E}^+(\lambda)$, since $\overline{\Psi_\N^+(\lambda,x)}$ is exponentially increasing as $x \to +\infty$ or $x \to -\infty$ for non-real $\lambda$. The idea is to use the half-plane representation in $\rmH_\W$ and $\rmH_\E$ respectively. Consider the case of $\Hphi_{\N,\W}^+(\lambda)$. Denoting $u_\W$ the restriction of $u$ to $\rmH_\W$, considered as a function of the local coordinates $(\x_\W,\y_\W)$, the half-plane representation \cref{eq: HPR} writes as
\begin{equation}
	u_\W(\x_\W,\y_\W) = \sum_\pm \int_{\R^+} \Hphi^\pm_\W(\mu)\,
	\Psi^\pm_\W(\mu,\x_\W) \,\e^{-\sqrt{\mu}\y_\W}\,\rho^\pm_\W(\mu)\,\d\mu,
	\label{eq: HPR in HW}
\end{equation}
for all $(\x_\W,\y_\W) \in \R \times [0,+\infty)$. According to \cref{eq: coord locales right angle}, for all $x \in (-\infty,a_\W),$ we have $\varphi_\N(x) = u(x,0) = u_\W(-b_\W,-x+a_\W)$. Hence, for $\lambda\in(-\k_{\N,+}^2,+\infty)$,
\begin{equation*}
	\Hphi_{\N,\W}^+(\lambda) = \int_{-\infty}^{a_\W} 
	\left( \sum_\pm \int_{\R^+} \Hphi^\pm_\W(\mu)\, 
	\Psi^\pm_\W(\mu,-b_\W) \,\e^{-\sqrt{\mu}(-x+a_\W)}\,\rho^\pm_\W(\mu)\,\d\mu \right) \overline{\Psi_\N^+(\lambda,x)}\,\d x .
\end{equation*}
In order to use Fubini's theorem, note that thanks to \cref{bound Psi}, we have
\begin{align*}
	& \int_{-\infty}^{a_\W} 
	\int_{\R^+} \Big| \Hphi^\pm_\W(\mu)\, 
	\Psi^\pm_\W(\mu,-b_\W) \,\e^{-\sqrt{\mu}(-x+a_\W)}\, \overline{\Psi_\N^+(\lambda,x)}\Big|\,\rho^\pm_\W(\mu)\,\d\mu\,\d x \\
	& \hspace{2cm} \leq 4 \int_{\R^+} 
	\left( \int_{-\infty}^{a_\W} \e^{-\sqrt{\mu}(-x+a_\W)}\,\,\d x \right) 
	\big| \Hphi^\pm_\W(\mu)\big|\, \rho^\pm_\W(\mu)\,\d\mu \\
	& \hspace{2cm} = 4 \int_{\R^+} 
	\frac{\big| \Hphi^\pm_\W(\mu)\big|}{\sqrt{\mu}} \, \rho^\pm_\W(\mu)\,\d\mu,
\end{align*}
which is bounded, according to \cref{eq: HPR prop phi} with $s=-1/2$. Switching the integrals in the above expression of $\Hphi_{\N,\W}^+(\lambda)$ and using the expression \cref{eq: def fpg} of $\Psi_\N^+(\lambda,x)$ then yields
\begin{multline*}
	\Hphi_{\N,\W}^+(\lambda) =  \\ \sum_\pm  
	\int_{\R^+}  T^+_\N(\lambda) \left( \int_{-\infty}^{a_\W}
	\e^{-\sqrt{\mu}(-x+a_\W) + \i\beta^-_\N(\lambda)\,x}\,\d x \right)
	\Hphi^\pm_\W(\mu)\,	\Psi^\pm_\W(\mu,-b_\W) \,\rho^\pm_\W(\mu)\,\d\mu,
\end{multline*}
since $T^+_\N(\lambda)$ and $\beta^-_\N(\lambda)$ are real on $(-\k_{\N,+}^2,+\infty)$. Computing the integral between parentheses, we finally obtain
\begin{equation}
	\Hphi_{\N,\W}^+(\lambda) =  T^+_\N(\lambda)\,\e^{\i\beta^-_\N(\lambda)\,a_\W} \,\sum_\pm  
	\int_{\R^+}  \frac{\Hphi^\pm_\W(\mu)\,\Psi^\pm_\W(\mu,-b_\W)}{\sqrt{\mu}+\i\beta^-_\N(\lambda)}
	\, \rho^\pm_\W(\mu)\,\d\mu.
	\label{eq: Tphi NW final}
\end{equation}
Now, we can use again Morera's theorem to conclude that $\Hphi_{\N,\W}^+(\lambda)$ has an analytic continuation in $\mathbb{D}_\N$, since $\sqrt{\mu}+\i\beta^-_\N(\lambda)$ cannot vanish for $\mu > 0$ and $\lambda \in \mathbb{D}_\N$.

We can proceed in a similar way for $\Hphi_{\N,\E}^+(\lambda)$. Using the half-plane representation in $\rmH_\E$, we obtain instead of \cref{eq: Tphi NW final}
\begin{equation*}
	\Hphi_{\N,\E}^+(\lambda) = \sum_\pm  \int_{\R^+} \left( 
	\frac{\e^{\i\beta^+_\N(\lambda)\,a_\E}}{\sqrt{\mu}-\i\beta^+_\N(\lambda)} + 
	R^+_\N(\lambda)\,\frac{\e^{-\i\beta^+_\N(\lambda)\,a_\E}}{\sqrt{\mu}+\i\beta^+_\N(\lambda)} \right)
	\,\Hphi^\pm_\E(\mu)\,\Psi^\pm_\E(\mu,b_\E) \,\rho^\pm_\E(\mu)\,\d\mu,
\end{equation*}
which also extends to an analytic function of $\lambda$ in $\mathbb{D}_\N$ thanks to \cref{prop: fpg}. This completes the proof.

\subsection{General angle configuration} 
\label{subsection: General angle}
We only detail here the items of the proof which differ from the right-angle configuration presented above. Indeed, the ideas are exactly the same, but the calculations are more complicated, since the relations between the local and global coordinate systems are given now by the general form \cref{eq: coord locales} with $(\theta_\W,\theta_\E ) \neq (+\pi/2,-\pi/2)$ (see \cref{fig: cas non hom/4k}, right). These relations show in particular that the intersection points of $\Sigma_\W \cap \Sigma_\N$ and $\Sigma_\E \cap \Sigma_\N$ are located respectively at $(a_{\N\W},0)$ and $(a_{\N\E},0)$ where
\begin{equation*}
	a_{\N\W} := a_\W - b_\W \cot\theta_\W
	\quad\text{and}\quad 
	a_{\N\E} := a_\E - b_\E \cot\theta_\E.
\end{equation*} 
Therefore the splitting \cref{eq:split TphiN} of $\Hphi_\N^+(\lambda)$ becomes here
\begin{multline*}
	\Hphi_\N^+(\lambda)= \\
	\underbrace{\int_{-\infty}^{a_{\N\W}}\varphi_\N(x)\,\overline{\Psi_\N^+(\lambda,x)}\,\d x}_{ \displaystyle =:\Hphi_{\N,\W}^+(\lambda)}
	+ \underbrace{\int_{a_{\N\W}}^{a_{\N\E}}\varphi_\N(x)\,\overline{\Psi_\N^+(\lambda,x)}\,\d x}_{ \displaystyle =:\Hphi_{\N,0}^+(\lambda)}
	+ \underbrace{\int_{a_{\N\E}}^{+\infty}\varphi_\N(x)\,\overline{\Psi_\N^+(\lambda,x)}\,\d x}_{ \displaystyle =:\Hphi_{\N,\E}^+(\lambda)},
\end{multline*}
As previously, the second integral $\Hphi_{\N,0}^+(\lambda)$ has an analytic continuation in $\mathbb{D}_\N$. But as we will see, this is no longer true for the two other integrals. We focus on $\Hphi_{\N,\W}^+(\lambda)$. The half-plane representation \cref{eq: HPR in HW} is still valid, but \cref{eq: coord locales} shows now that
\begin{align*}
	&\forall x \in (-\infty,a_{\N\W}),\quad 
	\varphi_\N(x)  = u(x,0) = u_\W\big(\x_\W(x),\y_\W(x) \big)
	\quad\text{where}\quad \\
	&\quad\quad \x_\W(x)  := (x-a_\W)\cos \theta_W-b_\W\sin\theta_\W
	\quad\text{and}\quad
	\y_\W(x) := (a_{\N\W}-x)\sin\theta_W.
\end{align*}
Therefore \cref{eq: HPR in HW} yields
\begin{equation*}
	\Hphi_{\N,\W}^+(\lambda) = \int_{-\infty}^{a_{\N\W}} 
	\left( \sum_\pm \int_{\R^+} \Hphi^\pm_\W(\mu)\, 
	\Psi^\pm_\W\big(\mu,\x_\W(x) \big)
	\,\e^{-\sqrt{\mu}\,\y_\W(x)}\,\rho^\pm_\W(\mu)\,\d\mu \right) \overline{\Psi_\N^+(\lambda,x)}\,\d x.
\end{equation*}
From Fubini's theorem, we obtain
\begin{align}
	\Hphi_{\N,\W}^+(\lambda) & = \sum_\pm \int_{\R^+}
	 A^\pm(\lambda,\mu)\,\Hphi^\pm_\W(\mu)\,\rho^\pm_\W(\mu)\,\d\mu
	\quad\text{where} \label{eq: TphiNW gen}\\
	A^\pm(\lambda,\mu) & := 
	 \int_{-\infty}^{a_{\N\W}} \e^{-\sqrt{\mu}\,\y_\W(x)}\,
	\Psi^\pm_\W\big(\mu,\x_\W(x)\big) \, \overline{\Psi_\N^+(\lambda,x)} \,\d x.
	\nonumber
\end{align}
Using the expression \cref{eq: def fpg} of $\Psi_\W^\pm$ and $\Psi_\N^+$ (note that $\x_\W(x) > 0$ for all $x \in (-\infty,a_{\N\W})$), as well as the fact that $\beta^+_\W(\mu)=\beta^-_\N(\mu)$ (see \cref{eq: k two-layer}), the calculation of the latter integral yields
\begin{align*}
	A^-(\lambda,\mu) 
	& = \frac{T^-_\W(\mu)\,\e^{+\i\beta^-_\N(\mu)\,\x_\W(a_{\N\W})}}
	{D^+(\lambda,\mu)}
	\,T^+_\N(\lambda)\,\e^{+\i\beta^-_\N(\lambda)\,a_{\N\W}} \\
	A^+(\lambda,\mu) 
	& = \left( 
	\frac{\e^{-\i\beta^-_\N(\mu)\,\x_\W(a_{\N\W})}}
	{D^-(\lambda,\mu)}
	+ \frac{R^+_\W(\mu)\,\e^{+\i\beta^-_\N(\mu)\,\x_\W(a_{\N\W})}}
	{D^+(\lambda,\mu)}
	\right)  T^+_\N(\lambda)\,\e^{+\i\beta^-_\N(\lambda)\,a_{\N\W}}
\end{align*}
where we have denoted
\begin{equation*}
	D^\pm(\lambda,\mu) :=
	\sqrt{\mu}\,\sin\theta_W \pm \i\beta^-_\N(\mu)\cos\theta_W + \i\beta^-_\N(\lambda).
\end{equation*}

For a fixed $\mu > 0$, the function $A^\pm(\lambda,\mu)$ has an analytic continuation in any complex subdomain of $\mathbb{D}_\N$ (see \cref{eq:def-DN}) in which the denominators $D^-(\lambda,\mu)$ and $D^+(\lambda,\mu)$ do not vanish. On the one hand, $D^-(\lambda,\mu)$ cannot vanish, for 
\begin{equation*}
	\forall \lambda\in \mathbb{D}_\N, \forall \mu > 0, \quad
	\Im\big( D^-(\lambda,\mu) \big) = 
	-\beta^-_\N(\mu)\cos\theta_W + \Re\big( \beta^-_\N(\lambda)\big) > 0,
\end{equation*}
since $\cos\theta_W \leq 0$ and $\Re\big( \beta^-_\N(\lambda)\big) > 0$ for all $\lambda\in \mathbb{D}_\N$ (according to our choice of the complex square root, see the proof of \cref{prop: fpg}). On the other hand, 
\begin{align*}
	D^+(\lambda,\mu)=0 
	& \Longleftrightarrow \sqrt{\lambda+\k_{\N,-}^2} = -\sqrt{\mu+\k_{\N,-}^2}\,\cos\theta_\W + 
	\i \sqrt{\mu}\,\sin\theta_\W \\
	& \Longleftrightarrow \lambda = 
	\left(\mu\,\cos(2\theta_\W) - \k_{\N,-}^2\,\sin^2\theta_\W\right)
	- \i\,\left(\sqrt{\mu}\,\sqrt{\mu+\k_{\N,-}^2}\,\sin(2\theta_\W) \right).
\end{align*}
The latter equation defines a curve $\Lambda_{\N,\W}$ parameterized by $\mu$ in the upper complex half-plane $\Im\,\lambda > 0$ (since $\sin(2\theta_\W)<0$, see \cref{eq: theta W E}), which reaches the real axis at $\lambda_{\N,\W} := -\k_{\N,-}^2\,\sin^2\theta_\W$ when $\mu = 0$. As a consequence, going back to \cref{eq: TphiNW gen}, we deduce from Morera's theorem that $\Hphi_{\N,\W}^+(\lambda)$ has an analytic continuation in $\mathbb{D}_\N \setminus\Lambda_{\N,\W}$ .

We proceed in the same way to deal with $\Hphi_{\N,\E}^+(\lambda)$. The only differences in the computations is that we use now the expression of $\Psi_\N^+(\lambda,x)$ for $x > 0$ (see \eqref{eq: def fpg}) and $\Psi^\pm_\E$ instead of $\Psi^\pm_\W$. We are then faced with the possible cancellation  of four denominators, which leads us to solve the following four equations:
\begin{equation*}
	\sqrt{\mu}\,\sin\theta_\E \pm \i\beta^+_\N(\mu)\cos\theta_E \pm \i\beta^+_\N(\lambda) = 0,
\end{equation*}
where both $\pm$ are independent. As above, only two of them have a solution, which yields
\begin{equation*}
	\lambda = 
	\left(\mu\,\cos(2\theta_\E) - \k_{\N,+}^2\,\sin^2\theta_\E\right)
	\pm \i\,\left(\sqrt{\mu}\,\sqrt{\mu+\k_{\N,+}^2}\,\sin(2\theta_\E) \right).
\end{equation*}
Instead of one curve, we then obtain two curves  $\Lambda_{\N,\E}^\pm$ which are complex-conjugate to each other and reach the real axis at $\lambda_{\N,\E} := -\k_{\N,+}^2\,\sin^2\theta_\E$. We conclude again that $\Hphi_{\N,\E}^+(\lambda)$ has an analytic continuation in $\mathbb{D}_\N \setminus \big(\Lambda_{\N,\E}^- \cup \Lambda_{\N,\E}^+\big)$.

To sum up, $\Tphi_{\N}^+(\lambda)$ has an analytic continuation in $\mathbb{D}_\N \setminus \big(\Lambda_{\N,\W} \cup \Lambda_{\N,\E}^- \cup \Lambda_{\N,\E}^+\big)$. The curves are represented in \cref{fig: analyticity curves}. We are interested in the analyticity of $\Tphi_{\N}^+(\lambda)$ in the gray area:  as $\Tphi_{\N}^+(\lambda)$ vanishes on the interval $(\max(\lambda_{\N,\W},\lambda_{\N,\E}),0)$, it implies that $\Tphi_{\N}^+(\lambda)$ also vanishes on $(0,+\infty)$.
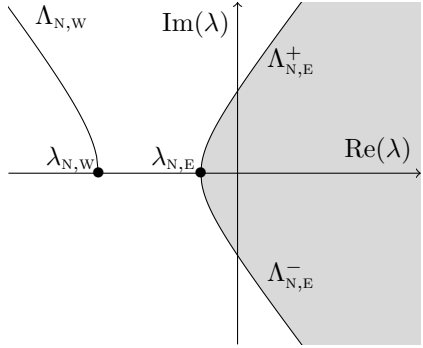
\begin{figure}[t]
	\centering
    \begin{tikzpicture}[]
		\begin{axis}[scale=.8,
				axis on top,
				xmin=-100,xmax=80,
				ymin=-100,ymax=100,
				ticks=none,
				xlabel={$\Re(\lambda)$},         
				ylabel={$\Im(\lambda)$},
				y label style={at={(0.35,1)}},
				]
			\addplot [name path=p,domain=0:300,samples=500]({x*cos(2*150)-(8*sin(150))^2},{sin(2*150)*sqrt(x*(x+64))}); 
				
			\addplot [name path=n,domain=0:300,samples=500]({x*cos(2*150)-(8*sin(150))^2},{-sin(2*150)*sqrt(x*(x+64))});
			
			\addplot [fill=gray!30]fill between[of=p and n];

			\addplot [domain=0:300,samples=200]({x*cos(2*120)-(9*sin(120))^2},{-sin(2*120)*sqrt(x*(x+81))});
			
			\node at (axis cs:{-64*(sin(150))^2},0) {{$\bullet$}};
			\node at (axis cs:{-81*(sin(120))^2},0) {{$\bullet$}};
			\node at (axis cs:{-64*(sin(150))^2-12},8) {{$\lambda_{\N,\E}$}};
			\node at (axis cs:{-81*(sin(120))^2-12},8) {{$\lambda_{\N,\W}$}};

			\node at (axis cs:{-49*(sin(60))^2+60},65) {{$\Lambda_{\N,\E}^+$}};
			\node at (axis cs:{-49*(sin(60))^2+60},-60) {{$\Lambda_{\N,\E}^-$}};
			\node at (axis cs:{-49*(sin(120))^2-40},90) {{$\Lambda_{\N,\W}$}};
		\end{axis}
	  \end{tikzpicture}\caption{Relevant domain of analyticity of $\Tphi_\N^+$ (gray area) for $\k_{\N,+}=10$, $\k_{\N,-}=9$, $\theta_\E=-5\pi/6$, $\theta_\W=2\pi/3$.}
	  \label{fig: analyticity curves}
\end{figure}

Finally, the same arguments apply to $\Tphi_{\N}^-(\lambda)$, so that \cref{Theorem: main} is now proved for a junction of two-layered media in the general angle configuration.

\begin{remark}
	Notice that when $(\theta_\W,\theta_\E)$ tends to $(+\pi/2,-\pi/2)$, the union of the curves represented in \cref{fig: analyticity curves} formally tends to the half-line $(-\infty,-\min(\k_{\N,-}^2,\k_{\N,+}^2)]$ on the real axis. We then recover the domain of analyticity $\mathbb{D}_\N$ of the right-angle configuration (see \cref{eq:def-DN}).
\end{remark}

\section{Extension to the general stratification case}
\label{section : general case}
We consider now the junction of general stratified media as described in \cref{setting pb}. We follow the same steps as previously, first focusing on a half-plane, explaining how the generalized Fourier transform is written in that case, detailing the properties of the half-plane representation and then using them to derive \cref{Theorem: main}.

\subsection{Generalized Fourier transform}\label{susbection: intro guided modes}
We consider again the unbounded selfadjoint operator $A$ defined by \cref{def: A} where $\k$ now  satisfies the general assumptions of \cref{setting pb}: it is a bounded function which is constant outside a bounded interval, more precisely $\k(\x) = \kpm $ if $\pm\x > \pm \x^\pm$. A main difference with the two-layered medium is that, in addition to the generalized eigenfunctions appearing previously, the operator may now have $L^2$ eigenfunctions. This means that there may exist pairs $(\lambda,\Psi) \in \R \times H^2(\R)\backslash\{0\}$ such that
\begin{equation}\label{eigen eq}
	-\partial_\x^2\Psi -\k^2\,\Psi = \lambda\,\Psi 
	\quad \text{in }\R.
\end{equation}
Using that $\Psi$ has to be decreasing where $\k$ is constant, we infer that there can exist such eigenpairs only for $\lambda<-\max(\km^2,\kp^2)$ otherwise non-trivial solutions to \cref{eigen eq} cannot be $L^2$. As this equation is one dimensional, all eigenvalues are of multiplicity one. Writing
\begin{equation*}
	\k_{\textsc M}^2:=\esssup\limits_{\x\in(\x^-,\x^+)}  (\k^2(\x)),
\end{equation*}
 and denoting $\Lambda^{\textsc P}$ the possibly empty point spectrum of $A$, multiplying \cref{eigen eq} by $\overline\Psi$ and integrating by parts, it follows that
 \begin{equation}\label{point spectrum}
    \Lambda^{\textsc P}\subset[-\k_{\textsc M}^2,-\max(\k_-^2,\k_+^2)]\subset \R^-.
\end{equation}
Then, an obvious necessary condition for the existence of eigenvalues is 
\begin{equation*}
    \k_{\textsc M}^2> \max(\k_-^2,\k_+^2).
\end{equation*}
 We can exhibit sufficient conditions ensuring the existence of at least one eigenpair. For instance, if $\km=\kp$, a sufficient condition \cite[Corollary 4.5.2]{Schechter-1981} is that $\k$ satisfies 
\begin{equation*}
	\int_\R(\k^2(\x)-\kp^2)\;\d\x>0.
\end{equation*}
Using the min-max principle, it is also possible to derive conditions on $\k$ ensuring the existence of any given number of eigenvalues. For instance, if $\k=\k_{\textsc M}$ on an interval of size $l$ and there exists a non-zero integer $n$ such that
\begin{equation*}
	\k_{\textsc M}^2-\max(\k_-^2,\k_+^2)>\frac{n^2\pi^2}{l^2},
\end{equation*}
then there exist at least $n$ eigenvalues.

Finally, it is proven in \cite[Theorem 4.7.1]{Schechter-1981} that there is at most a finite number of eigenvalues.

In what follows, for the sake of conciseness, we consider the situation where the point spectrum is not empty. The empty case then follows easily since we will see in \cref{prop phi gen} that the existence of eigenfunctions does not impact the core of the proof. We denote by $N\geq1$ the number of eigenvalues $(\lambda^n)_{n=1}^N$ and for all $n\in\llbracket1, N\rrbracket$, by $\Psi^n:\R\longrightarrow\C$ an associated eigenfunction normalized in $L^2(\R)$ associated to $\lambda^n$.

As previously, we also consider generalized eigenfunctions (which are bounded but not $L^2$ solutions to \cref{eigen eq}) for $\lambda>-\max(\km^2,\kp^2)$. The dimension of the space of bounded solutions is once again 1 if $-\max(\km^2,\kp^2) \leq \lambda < -\min(\km^2,\kp^2)$ and 2 if $\lambda \geq -\min(\km^2,\kp^2)$. The generalized eigenfunctions $\Psi^\pm : \Lambda^\pm\times\R \longrightarrow \C$ are now more complicated to describe. They are uniquely defined (see  \cref{proof analyticity}) as solutions to \cref{eigen eq} that have the following form outside of the interval $(\x^-,\x^+)$: $\forall (\lambda,\x) \in \Lambda^\pm\times\R,$
\begin{align}
	&\Psi^\pm(\lambda,\x) := \left\{\begin{array}{ll}
		\e^{\mp\i\beta^\pm(\lambda)\,\x} + R^\pm(\lambda)\,\e^{\pm\i\beta^\pm(\lambda)\,\x} &\text{if } \pm \x >\pm\x^\pm, \\[1mm]
		T^\pm(\lambda)\,\e^{\mp\i\beta^\mp(\lambda)\,\x} & \text{if }\mp \x>\mp\x^\mp,
	\end{array}\right.\label{eq: expression generale psis}
\end{align}  
for some complex coefficients $T^\pm(\lambda)$, $ R^\pm(\lambda)$. If $\k$ is piecewise constant, we can still retrieve explicit expressions for the generalized eigenfunctions in $(\x^-,\x^+)$ and for the associated coefficients $T^\pm(\lambda)$ and $R^\pm(\lambda)$ \cite[Section 2.3]{Ott-2017}, but in general, it is no longer possible. 

We are now ready to give the diagonalization theorem for the general stratification case. A proof can be found in \cite[Section 2]{Ott-2017}. The spectral space adapted to this situation writes as follows:
\begin{equation*}
    \widetilde\H :=L^2(\Lambda^+,\,\rho^+ )\times L^2(\Lambda^-,\,\rho^-)\times \C^N,
\end{equation*} and is  endowed with the norm
\begin{equation*}
    \forall \widehat\varphi:=(\widehat\varphi^+,\widehat\varphi^-,(\widehat\varphi^n)_{n=1}^N)\in\widetilde{\H},\quad||\widehat\varphi||_{\widetilde{\H}}^2:=||\widehat\varphi^+||_{L^2(\Lambda^+,\rho^+)}^2+||\widehat\varphi^-||_{L^2(\Lambda^-,\rho^-)}^2+\sum_{n=1}^N |\widehat\varphi^n|^2,
\end{equation*}
where we used the same notations as \cref{eq:def Lambda_pm} and \cref{eq: def rho} since the continuous spectrum is unchanged.
\begin{theorem}\label{theoreme fourier generalise cas gen}
    The generalized Fourier transform $\mathcal{F}$ defined for all $\varphi\in L^1(\R)\cap L^2(\R)$ by 
   \begin{align*}
	& \mathcal{F}\varphi := (\mathcal F^- \varphi,\mathcal F^+ \varphi,(\mathcal F^n \varphi)_{n=1}^N)\text{ with, for }\lambda\in\Lambda^\pm\text{ and }n\in\llbracket 1,N\rrbracket,\\
       &\mathcal F^\pm \varphi(\lambda):=\int_\R\varphi(\x)\overline{\Psi^\pm(\lambda,\x)}\;\d\x\quad\text{and}
      \quad \mathcal F^n \varphi := \int_\R\varphi(\x)\overline{\Psi^n(\x)}\;\d\x,
   \end{align*} 
   extends by density to a unitary operator from $L^2(\R)$ to $\widetilde\H$. It diagonalizes $A$ on $L^2(\R)$, in the sense  that
   \begin{equation*}
	   \forall \varphi\in H^2(\R),\quad  A\varphi=\calF^{-1} \lambda\,\calF\varphi, 
   \end{equation*}
   where the notation $\lambda$ stands for the operator of multiplication by $\lambda$ in $\widetilde\H$. Moreover, the inverse transform $\mathcal{F}^{-1}$ writes explicitely as follows: for all $(\widetilde\varphi^-,\widetilde\varphi^+,(\widetilde\varphi^{\,n})_{n=1}^N)\in \widehat{\H}$ such that $\widehat\varphi^\pm\in L^1(\Lambda^\pm,\rho^\pm)$, for all $\x\in\R,$
   \begin{equation*}
   \mathcal{F}^{-1} [(\widetilde\varphi^-,\widetilde\varphi^+,(\widetilde\varphi^{\,n})_{n=1}^N)](\x) = \sum_{n=1}^N\Psi^n(\x)\widetilde\varphi^{\,n}+ \sum_{\pm}\int_{- \kpm^2}^{+\infty}\Psi^\pm(\lambda,\x)\widetilde\varphi^{\,\pm}(\lambda)\rho^\pm(\lambda)\d\lambda.
   \end{equation*}
   \end{theorem} 
\begin{remark}\label{remark: Psi gen bounded}By construction \cite[Lemma 2.4.8]{Ott-2017}, the $\Psi^\pm$'s are uniformly bounded on $\Lambda^\pm\times\R$ so that the formulas make sense for $L^1$ functions (as mentioned in \cref{rem Fourier gen L2-L1}). \end{remark}

\subsection{Half-space representation}
Let us point out that compared to the two-layered situation, two difficulties arise when generalizing the proof of \cref{Theorem: main}. Firstly, we now have to handle the $L^2$ eigenfunctions in addition to the generalized ones. Secondly, since we do not have explicit expressions of the generalized eigenfunctions anymore, it is more technical to derive the necessary analyticity properties.

We start by handling the first difficulty and we show that since all the eigenvalues are negative, the eigenfunctions do not contribute to the half-plane representation of $L^2$ solutions of the Helmholtz equation. 
More precisely, the following proposition states that the half-plane representation of an $L^2$ solution to \cref{eq Helmholtz} still writes as in \cref{eq: HPR} without the components associated to the eigenfunctions $\Psi^n$. 
\begin{proposition}\label{prop phi gen}
    Considering $\rmH$ a half-plane as defined in \cref{def half plane eps}, suppose that $u\in L^2(\rmH)$ satisfies
        \begin{equation}\label{eq: helmholtz prop phi gen}
             -\Delta u - \k^2 u = 0  \text{ in } \rmH
        \end{equation}
      in the distributional sense where $\k$ is defined as in \cref{susbection: intro guided modes}. 
	  
	  Then, writing $(\widetilde\varphi^-,\widetilde\varphi^+,(\widetilde\varphi^{\,n})_{n=1}^N):=\mathcal{F}\,u(\cdot,0)$, we have
	  \begin{equation}
	  \forall n\in \llbracket1,N\rrbracket,\quad\widetilde\varphi^{\,n}=0,\end{equation}
	  and, 
    \begin{equation} 
        \forall (\x, \y) \in \R\times[0,+\infty),\quad u(\x, \y) = \sum_{\pm}\int_{\R^+}\widetilde\varphi^{\,\pm}(\lambda)\Psi^\pm(\lambda,\x)\e^{-\sqrt{\lambda}\y}\rho^\pm(\lambda)\d\lambda,
    \end{equation}
    with $\widetilde\varphi^\pm$ satisfying \cref{eq: HPR prop phi}.
\end{proposition}
\begin{proof}
Following the same reasoning as in the proof of \cref{prop: HPR}, we can define a family of $L^2$ functions $\widetilde u^{\,n}(\y):=\mathcal{F}^{\,n} u(\cdot,\y)$ for almost every $\y\in(-\varepsilon,+\infty)$ and for $n\in\llbracket1,N\rrbracket$. Then applying the partial generalized Fourier transform in the $\x$-direction to the Helmholtz equation \cref{eq: helmholtz prop phi gen}, we have that for all $n\in\llbracket1,N\rrbracket$,
\begin{equation*}
    -{\partial^2_\y\widetilde u^{\,n}} + \lambda^n \widetilde u^{\,n} = 0\text{ on }(-\varepsilon,+\infty)
 \end{equation*}
 in the distributional sense. Then, there exist constants $ \widetilde A^{\,n}$, $ \widetilde B^{\,n}\in\C$ such that
\begin{equation*}
    \widetilde u^{\,n}(\y) = \widetilde A^{\,n} e^{-(\y+\varepsilon)\sqrt{\lambda^n}} + \widetilde B^{\,n} e^{+(\y+\varepsilon)\sqrt{\lambda^n}}.
\end{equation*}
 Since $\widetilde u^{\,n}\in L^2(-\varepsilon,+\infty)$ and since all the eigenvalues are negative by virtue of \cref{point spectrum}, we have that $\widetilde A^{\,n}=0$ and $\widetilde B^{\,n} = 0$. 

The properties on $\widetilde\varphi^\pm$ are then derived as in \cref{prop: HPR}. 
\end{proof}

The conclusion of \cref{cor: HPR} still holds for the hypotheses of \cref{prop phi gen} by virtue of \cref{remark: Psi gen bounded}
.

To proceed as in the proofs detailed in \cref{sec: two-layer}, we need an additional analyticity  result on the general eigenfunctions for which we cannot rely on explicit expressions anymore. Recalling $\mathbb D$ defined in \cref{eq def D}, we have the following.
\begin{proposition}\label{prop analyticity}
The mappings  $\lambda\longrightarrow R^\pm(\lambda)$ and 
$\lambda\longrightarrow T^\pm(\lambda)$ 
corresponding to the coefficients of the generalized eigenfunctions outside $(\x^-,\x^+)$ as defined in \cref{eq: expression generale psis} have meromorphic continuations in $\mathbb{D}_\N$. Moreover, for any $\x\in\R$, the mappings to the generalized eigenfunctions $\lambda\longrightarrow \Psi^\pm(\lambda,\x)$ 
 also have meromorphic continuations in $\mathbb{D}_\N$ with poles independent from $\x$. Finally, the meromorphic continuations
of $\Psi^\pm$ 
are bounded in any compact set of $\mathbb D\backslash P\times\R$ where $P$ denotes their set of poles.
\end{proposition}
\begin{proof}
	The idea is to prove that the $\Psi^\pm(\cdot,\x)$'s are meromorphic for all $\x\in[\x^-,\x^+]$ which directly yields the same property for the coefficients $R^\pm(\lambda)$ and $T^\pm(\lambda)$ and then for any $\x\in\R$. The proof is given in \cref{proof analyticity}.
\end{proof}

\subsection{Proof of \cref{Theorem: main}}
    We use the same notations as in the previous proofs. We can once again derive the following expression:
	\begin{multline*}
		\Hphi_\N^+(\lambda)= \\
		\underbrace{\int_{-\infty}^{a_{\N\W}}\varphi_\N(x)\,\overline{\Psi_\N^+(\lambda,x)}\,\d x}_{ \displaystyle =:\Hphi_{\N,\W}^+(\lambda)}
		+ \underbrace{\int_{a_{\N\W}}^{a_{\N\E}}\varphi_\N(x)\,\overline{\Psi_\N^+(\lambda,x)}\,\d x}_{ \displaystyle =:\Hphi_{\N,0}^+(\lambda)}
		+ \underbrace{\int_{a_{\N\E}}^{+\infty}\varphi_\N(x)\,\overline{\Psi_\N^+(\lambda,x)}\,\d x}_{ \displaystyle =:\Hphi_{\N,\E}^+(\lambda)}.
	\end{multline*}
By virtue of the definitions of the three stratified half-planes detailed in the introduction, we necessarily have $[\x_\N^-,\x_\N^+]\subset[a_{\N\W},a_{\N\E}]$. Then, $\Psi_\N^+$ can be expressed semi-explicitely as in \cref{eq: expression generale psis} in both $\Hphi_{\N,\W}^+$ and $\Hphi_{\N,\E}^+$. The fact that there is no explicit expression of $\Psi_\N^+$ in $(a_{\N\W},a_{\N\E})$ does not matter. Indeed, since $\Hphi_{\N,0}^+$ is defined by an integral on a bounded interval, by virtue of \cref{prop analyticity}, we can use Morera's theorem as in \cref{sec: two-layer} to show that this term is meromorphic. Moreover, we also have $\x_\W(x)>\x_\W^+$ on $(-\infty,a_{\N\W})$ and $\x_\E(x)<\x_\E^-$ on $(a_{\N\E},+\infty)$ so that the $\Psi^+_\J$'s, $\J\in\{\E,\W\}$, both have semi-explicit expressions of the form \cref{eq: expression generale psis} where they are considered. This allows us to detail $\Hphi_{\N,\W}^+$ and $\Hphi_{\N,\E}^+$ as previously. Then, proceeding as in \cref{sec: two-layer} and using \cref{prop analyticity}, we get that $\widetilde\varphi_\N(\lambda)$ has a meromorphic extension to $\mathbb{D}_\N \setminus \big(\Lambda_{\N,\W} \cup \Lambda_{\N,\E}^- \cup \Lambda_{\N,\E}^+\big)$. Consequently,  we can define a neighbourhood of $(-\eta,+\infty)$, for some $\eta>0$, where $\widetilde\varphi_\N(\lambda)$ is analytic and conclude as previously.

\section{Conclusion}
In this paper, we proved that in media that can be described as the union of stratified half-planes, there does not exist non-trivial square-integrable solutions to the two-dimensional Helmholtz equation. The tools used in the proof led to considering angles between the stratifications greater than $\pi/2$. In particular, the absence of trapped modes at the junction of waveguides when the angle between at least two branches is smaller than $\pi/2$ is an open question. We conjecture that the result still holds and that this limitation is purely technical, due to the use of separation of variables in half-planes.

Interestingly, our result still holds without imposing that $k^2$ be positive everywhere. Indeed, our proof only requires the positivity of $\k_{\N,\pm}^2$, which means that $\k^2_{\E,+}$ and $\k_{\W,-}^2$ can be negative.
In other words, even if waves do not propagate in the southern quadrants of $\rmH_\E$ and $\rmH_\W$, the  presence of propagating waves in $\rmH_\N$ is sufficient to prevent the existence of  trapped modes. Let us point out that without information coming from the southern quadrants of $\rmH_\E$ and $\rmH_\W$, we cannot conclude. Indeed there exist non-trivial $L^2$ solutions to the Helmholtz equation in a half-plane, see \cite{Bonnet-Fliss-Hazard-Tonnoir-2011} for the homogeneous case.

Note that we could have considered instead of the Helmholtz equation \cref{eq Helmholtz} a more general equation of the type\begin{equation*}
	-\textrm{div}(\mu\nabla u)-\rho u = 0.
\end{equation*}
Our result remains valid as soon as $\mu$ and $\rho$ satisfy the hypotheses given on $k^2$ in this paper, provided that $\mu$ is bounded from below by a positive constant. As mentioned above, the case of negative $\rho$ in the southern part of the domain  is also allowed.

 Let us now consider the Helmholtz equation in a three-dimensional domain  $\Omega\times\R$ where $\Omega\subset\R^2$ and 
$k=k(x,y)$ independent from the third coordinate $z$ both satisfy the hypotheses of \cref{setting pb}. Our
result means that there cannot exist so-called guided waves, i.e. non-trivial solutions of the form
\begin{equation*}
	u(x,y,z)=\hat{u}(x,y)\e^{\i\xi z}
\end{equation*} with $\hat{u}\in L^2(\Omega)$  and $\xi\in\R$
if $\xi^2<\min(\k^2_{\N,+},\k^2_{\N,-})$. 

A natural extension would be to apply  our approach to three-dimensional domains that
are the union of half-spaces (five in a right-angle configuration and four otherwise), such that in each half-space, $k$  depends
only on the transverse coordinates. However, this extension is far from straightforward as adapting the proof raises several difficulties. 
In particular, \cref{cor: HPR} does not directly extend to higher dimensions. Again, it is possible that this limitation is only due to the technique that we use. Indeed, in the homogeneous case, this difficulty was circumvented by a tricky use of the two-dimensional result (see \cite{Bonnet-Fliss-Hazard-Tonnoir-2011}). Unfortunately, this approach does not extend to the heterogeneous case we are interested in here.

It could be interesting to investigate links with the results concerning the so-called Landis' conjecture \cite{fernándezbertolin2024landisconjecturesurvey} which also applies to the Helmholtz equation with non-constant coefficients in exterior domains. This conjecture provides an upper bound for the rate of exponential decay at infinity of trapped modes.

\appendix
\section{\nopunct}\label{proof analyticity}
In this appendix, we first prove the existence and uniqueness of the generalized eigenfunctions $\Psi^\pm(\lambda,\cdot)$, introduced as the bounded solutions to \cref{eigen eq} satisfying \cref{eq: expression generale psis} for $\lambda\in\Lambda^\pm$, and then we prove \cref{prop analyticity}.

To prove that the $\Psi^\pm$'s are uniquely defined, we notice that they can be defined as the solutions to \cref{eigen eq} which satisfy the following Robin conditions at $\x^\pm$.
For $\Psi^+$ we have,
\begin{equation}\label{Robin psi+}
	\begin{array}{l}
	\partial_\x\Psi^+(\lambda,\x^-)+\i\beta^-(\lambda)\Psi^+(\lambda,\x^-) = 0,\\
	\partial_\x\Psi^+(\lambda,\x^+)-\i\beta^+(\lambda)\Psi^+(\lambda,\x^+) = -2\i\beta^+(\lambda) \e^{-\i\beta^+(\lambda)\x^+},
	\end{array}
\end{equation}
and for $\Psi^-$  (notice that only the right-hand side changes),
\begin{equation}\label{Robin psi-}
	\begin{array}{l}\partial_\x\Psi^-(\lambda,\x^-)+\i\beta^-(\lambda)\Psi^-(\lambda,\x^-) = 2\i\beta^-(\lambda) \e^{\i\beta^-(\lambda)\x^+},\\
		\partial_\x\Psi^-(\lambda,\x^+)-\i\beta^+(\lambda)\Psi^-(\lambda,\x^+) = 0.
	\end{array}
\end{equation}
The idea is to show constructively that the generalized eigenfunctions can be defined as linear combinations of entire canonic solutions to an ODE and derive the different properties from this description. 

 Let us then introduce, for $\lambda$ in $\mathbb C$, $s(\lambda,\cdot)$ and $c(\lambda,\cdot)$ the two solutions in $C^1(\R)$ to
\begin{equation*}
	-\partial_\x^2 u(\lambda,\x) - (\k^2(\x)+\lambda) u(\lambda,\x) = 0\quad\text{in }\R
\end{equation*}
in the distributional sense, such that
\begin{align*}
            c(\lambda,\x^-) = 1 \quad\text{and}\quad \partial_\x c(\lambda,\x^-)=0,\\
			s(\lambda,\x^-) = 0 \quad\text{and}\quad\partial_\x s(\lambda,\x^-)=1.
\end{align*}

It is well-known that for all $\x\in\R$, $s(\cdot,\x)$ and $c(\cdot,\x)$ are entire functions (see for instance \cite[Section 4.2]{Kirsch-2011}). 

Since the $\Psi^\pm$'s satisfy the same equation as $c$ and $s$, there exist complex coefficients $A^\pm(\lambda)$ and $B^\pm(\lambda)$ such that
\begin{equation}\label{eq: decomposition Psi}
\forall\lambda\in\Lambda^\pm,\;\forall \x\in\R,\quad	\Psi^\pm(\lambda,\x)=A^\pm(\lambda)s(\lambda,\x)+B^\pm(\lambda)c(\lambda,\x).
\end{equation}
The Robin conditions \cref{Robin psi+,Robin psi-} on $\Psi^\pm$ then yield the following system
\begin{align*}
	\begin{pmatrix}
		1&\i\beta^-(\lambda)\\\partial_\x s(\lambda,\x^+)-\i\beta^+(\lambda){s}(\lambda,\x^+) & \partial_\x c(\lambda,\x^+)-\i\beta^+(\lambda){c}(\lambda,\x^+)
	\end{pmatrix}\begin{pmatrix}
		A^\pm(\lambda)\\B^\pm(\lambda)
	\end{pmatrix}& ={L^\pm}(\lambda),
\end{align*}
where ${L^\pm}(\lambda)$ only depend on $\beta^\pm(\lambda)$.

To prove that the $\Psi^\pm$'s are uniquely defined for $\lambda\in\Lambda^\pm$, it suffices to prove that the determinant of the above matrix never vanishes. We give the proof for  $\lambda\in(-\min(\kp^2,\km^2),+\infty)$ where both $\beta^\pm(\lambda)$ are real. It can easily be extended to $\lambda\in(-\max(\kp^2,\km^2),-\min(\kp^2,\km^2))$ where one of the $\beta^\pm(\lambda)$'s is imaginary. 

Let us assume by contradiction that the determinant vanishes for some $\lambda\in(-\min(\kp^2,\km^2),+\infty)$. This yields, as both $s(\lambda,\cdot)$ and $c(\lambda,\cdot)$ are real-valued for all $\lambda\in\R$,
\begin{equation}\label{matrix determinant}
	\forall \lambda\in(-\min(\kp^2,\km^2),+\infty),\;\left\{\begin{array}{l}\displaystyle
		s(\lambda,\x^+)=\frac{\partial_\x c(\lambda,\x^+)}{\beta^-(\lambda)\beta^+(\lambda)},\\[2mm]\displaystyle
		\partial_\x s(\lambda,\x^+)=-\frac{\beta^+(\lambda)c(\lambda,\x^+)}{\beta^-(\lambda)}	.
	\end{array}\right.
\end{equation}
Consider now the wronskian of $c$ and $s$ defined by 
\begin{equation*}
\forall \lambda\in\mathbb C,\;\forall\x\in\R,\quad	W(\lambda,\x):=c(\lambda,\x) \partial_\x s(\lambda,\x)-s(\lambda,\x) \partial_\x c(\lambda,\x).
\end{equation*}
 It is a constant function of $\x$. By definition of $c$ and $s$, $W(\lambda,\x^-)=1$, but \cref{matrix determinant} yields
\begin{equation*}
	\forall \lambda\in(-\min(\kp^2,\km^2),+\infty),\quad W(\lambda,\x^+) = -\frac{\beta^+(\lambda)(c(\lambda,\x^+))^2}{\beta^-(\lambda)}-\frac{(\partial_\x c(\lambda,\x^+))^2}{\beta^-(\lambda)\beta^+(\lambda)}<0
\end{equation*}which gives a contradiction.

Now, we prove \cref{prop analyticity}. Notice that the coefficients of the matrix and the right-hand term can analytically be extended to $\mathbb D$. Since the determinant can vanish on $\mathbb D$ outside of $\Lambda^\pm$, the solutions of the linear system $(A^\pm(\lambda),B^\pm(\lambda))$ have meromorphic continuations and have poles whenever the determinant vanishes. It follows that for every $\x\in[\x^-,\x^+]$, the $\Psi^\pm(\cdot,\x)$'s now have a meromorphic continuation to $\mathbb D$ with poles independent from $\x$.

Finally, the fact that the $\Psi^\pm$'s are bounded on any compact of $\mathbb D \times \R$ where they are analytic follows directly from their decomposition \cref{eq: decomposition Psi} since the mappings $c$ and $s$ are bounded on any compact of $\mathbb C\times \R$ \cite[Theorem 4.5]{Kirsch-2011}.

\begin{remark}\label{Annex: analysis arguments} 
	The existence, uniqueness and meromorphicity of the $\Psi^\pm$'s can also be derived from scattering theory using the analytic Fredholm theorem. The well-posedness of the problems defining $\Psi^\pm$ is a consequence of the energy conservation and the poles of the meromorphic extensions correspond to the so-called scattering resonances \cite{Zworski-scattering-resonances}.

\end{remark}
\bibliographystyle{siamplain}
\bibliography{biblio}
\end{document}